\documentclass[12pt, reqno]{amsart}
\makeatletter
\usepackage[margin=2cm]{geometry}
\usepackage{amsmath, amssymb, amsfonts, amstext, verbatim, amsthm, mathrsfs}
\usepackage[mathcal]{eucal}
\usepackage[all]{xy}
\usepackage[usenames]{color}
\usepackage{aliascnt} 
\usepackage[colorlinks=true,linkcolor=blue,citecolor=blue,urlcolor=blue,citebordercolor={0 0 1},urlbordercolor={0 0 1},linkbordercolor={0 0 1}]{hyperref} 
\usepackage{enumerate}
\usepackage{tikz}
\usetikzlibrary{backgrounds}
\theoremstyle{plain}

\newcommand{\refnewtheoremn}[4]{
\newaliascnt{#1}{#2}
\newtheorem{#1}[#1]{#3}
\aliascntresetthe{#1}
\expandafter\providecommand\csname #1autorefname\endcsname{#4}}

\newcommand{\refnewtheorem}[3]{\refnewtheoremn{#1}{#2}{#3}{#3}}

\def\makeCal#1{
\expandafter\newcommand\csname c#1\endcsname{\mathcal{#1}}}
\def\makeBB#1{
\expandafter\newcommand\csname b#1\endcsname{\mathbb{#1}}}
\def\makeFrak#1{
\expandafter\newcommand\csname f#1\endcsname{\mathfrak{#1}}}

\count@=0
\loop
\advance\count@ 1
\edef\y{\@Alph\count@}
\expandafter\makeCal\y
\expandafter\makeBB\y
\expandafter\makeFrak\y
\ifnum\count@<26
\repeat

\newtheorem{thm}{Theorem}[section]

\refnewtheorem{prethm}{thm}{Pre-theorem}

\refnewtheorem{lemma}{thm}{Lemma}
\refnewtheorem{claim}{thm}{Claim}
\refnewtheorem{cor}{thm}{Corollary}
\refnewtheorem{conj}{thm}{Conjecture}
\refnewtheorem{prop}{thm}{Proposition}
\refnewtheorem{quest}{thm}{Question}
\refnewtheorem{amplif}{thm}{Amplification}
\refnewtheorem{scholium}{thm}{Scholium}

\refnewtheorem{tablethm}{thm}{Table}
\refnewtheorem{assumption}{thm}{Assumption}
\refnewtheorem{conclusion}{thm}{Conclusion}
\refnewtheorem{problem}{thm}{Problem}
\refnewtheorem{ansatz}{thm}{Ansatz}

\theoremstyle{definition}
\refnewtheorem{remark}{thm}{Remark}
\refnewtheorem{remarks}{thm}{Remarks}
\refnewtheorem{defn}{thm}{Definition}
\refnewtheorem{notn}{thm}{Notation}
\refnewtheorem{const}{thm}{Construction}
\refnewtheorem{ex}{thm}{Example}
\refnewtheorem{fact}{thm}{Fact}
\refnewtheorem{recollection}{thm}{Recollection}

\newcommand{\fh}{\mathfrak{h}}

\newcommand{\Stab}{\operatorname{Stab}}

\renewcommand{\Im}{\operatorname{Im}}
\renewcommand{\Re}{\operatorname{Re}}
\newcommand {\Hom}{\operatorname{Hom}}
\newcommand {\Aut}{\operatorname{Aut}}

\newcommand{\DT}{\operatorname{DT}}
\newcommand{\ch}{\operatorname{ch}}
\newcommand{\Coh}{\operatorname{Coh}}

\newcommand{\height}{{H}}
\newcommand{\Li}{\operatorname{Li}}
\newcommand{\GV}{\operatorname{GV}}

\newcommand {\<}{\langle}
\renewcommand {\>}{\rangle}
\newcommand{\isom}{\cong}

\newcommand{\tensor}{\otimes}
\renewcommand{\O}{\mathscr{O}}

\newcommand{\D}{D}

\newcommand{\dual}{\vee}

\renewcommand{\b}{\,|\,}

\newcommand{\oomit}[1]{}

\usepackage{calligra}
\DeclareMathAlphabet{\mathcalligra}{T1}{calligra}{m}{n}
\DeclareFontShape{T1}{calligra}{m}{n}{<->s*[2.2]callig15}{}

\begin{document}

\title{ Riemann-Hilbert problems for the resolved conifold}
\author{Tom Bridgeland}
\date{}

\begin{abstract}{We study the Riemann-Hilbert problems of  \cite{RHDT} in the case of  the Donaldson-Thomas theory of the resolved conifold.  We give  explicit solutions in terms of the Barnes double and triple sine functions. We show that the $\tau$-function of  \cite{RHDT} is a non-perturbative partition function, in the sense  that its asymptotic expansion   coincides with  the   topological  closed string  partition function. } \end{abstract}

\maketitle



\section{Introduction}

In \cite{RHDT} we studied a class of Riemann-Hilbert problems arising naturally in Donaldson-Thomas theory. They involve piecewise holomorphic maps from the complex plane into an algebraic torus $(\bC^*)^n$ with prescribed discontinuities along a given collection of rays. These problems represent the conformal limit of the Riemann-Hilbert problems appearing in the work of  Gaiotto, Moore and  Neitzke \cite{GMN1}. The purpose of this paper is to give a detailed solution to the Riemann-Hilbert problems associated to the resolved conifold using a class of special functions related to Barnes' multiple gamma functions \cite{Barnes2,Barnes3,Barnes4}.  
We also compute the  $\tau$-function in the sense of \cite{RHDT}, and show that the  asymptotic expansion of $\log(\tau)$  reproduces the positive degree terms in  the genus expansion of the topological  string free energy. Our calculations thus suggest a new approach to defining non-perturbative partition functions in topological string theory. \smallskip

\subsection{BPS structures for the resolved conifold}

Let $X$ denote the resolved conifold: this is the  non-compact Calabi-Yau threefold which is the total space of the vector bundle $\O_{\bP^1}(-1)^{\oplus 2}$. Contracting the zero-section $C\subset X$ gives the threefold ordinary double point \[(x_1 x_2-x_3x_4=0)\subset \bC^4.\] The Riemann-Hilbert problems we shall consider arise from the Donaldson-Thomas  (DT) theory of the category of compactly-supported coherent sheaves on $X$.
They depend on a point  in the space
\[M=\big\{(v,w)\in \bC^2: w\neq 0 \text{ and }v+nw\neq 0\text { for all }n\in \bZ\big\}\subset \bC^2.\]
Mathematically speaking, as we recall in Appendix A, this is the space
 of stability conditions on the derived category $\operatorname{D^bCoh}(X)$, quotiented by the subgroup of the group of auto-equivalences  generated by spherical twists. From the physical standpoint, it can be thought of as the smallest unramified cover of the natural  $\bC^*$-bundle over the stringy K{\"a}hler moduli space on which central charges of branes are single-valued.
 
Associated to a point of $M$ is a collection of data which we referred to in \cite{RHDT} as a BPS structure. Mathematically it represents the output of unrefined DT theory applied to the  given stability condition. In physical terms it encodes the BPS invariants of the non-linear supersymmetric sigma model associated to the space $X$. It consists of

\begin{itemize}
\item[(i)]
The charge lattice $\Gamma_{\leq 1}=\bZ\beta\oplus \bZ\delta$  equipped with the zero skew-symmetric form $\<-,-\>=0$.\smallskip

\item[(ii)]The central charge: this is the group homomorphism
\[Z_{\leq 1}\colon \Gamma_{\leq 1}\to \bC, \qquad Z_{\leq 1}(a\beta+b\delta)=2\pi i(av+bw).\]

\item[(iii)]
The   nonzero BPS invariants 
\begin{equation}
\label{partiii2}\Omega(\gamma)=\begin{cases} 1 &\text{if }\gamma=\pm \beta+n\delta \text{ for some } n\in \bZ,\\
-2 &\text{if }\gamma=k\delta\text{ for some }k\in \bZ\setminus\{0\}.\end{cases}\end{equation}
\end{itemize}

The lattice $\Gamma_{\leq 1}$ is the natural receptacle for the Chern characters of compactly-supported sheaves on $X$. All such sheaves are supported in dimension $\leq 1$: the most important examples  are the line bundles $\O_C(n)$ supported on the zero-section $C\subset X$, and the skyscraper sheaves $\O_x$ supported at points $x\in X$. We choose the sign of the generator $\delta$ so that  \[\ch(\O_C(n))=\beta-n\delta, \qquad \ch(\O_x)=-\delta.\] The two cases in \eqref{partiii2} give the contribution to DT theory  from sheaves supported in dimension one and  zero respectively, and arise from extensions of the above-mentioned sheaves and their shifts. For the mathematical derivation of \eqref{partiii2}   we refer to \cite[Example 6.30]{JS}.

The form $\<-,-\>$, which in the general context of \cite{RHDT} is the Euler form of the relevant category, vanishes in this case because curves on a threefold have zero intersection number. This has the consequence that the BPS invariants $\Omega(\gamma)$ do not depend on the choice of point $(z,w)\in M$. 
It also implies that the Riemann-Hilbert problem associated to our BPS structure is trivial.
 To remedy this  we consider  the double of the above BPS structure. This involves replacing the lattice $\Gamma_{\leq 1}$ with the lattice
 \[\Gamma=\Gamma_{\leq 1} \oplus \Gamma_{\geq 2}, \qquad \Gamma_{\geq 2}:=\Gamma_{\leq 1}^\vee=\Hom_\bZ(\Gamma_{\leq 1},\bZ),\]
 equipped with the canonical non-degenerate integral skew-symmetric form, and extending the map $\Omega$ by zero. We can extend the map $Z_{\leq 1}$ via an arbitrary group homomorphism
\[Z_{\geq 2}\colon \Gamma_{\geq 2}\to \bC,\]
so that the space of possible doubled BPS structures becomes the cotangent bundle $T^*M$.

\subsection{The Riemann-Hilbert problem}

Introduce the twisted torus
\[\bT=\big\{g\colon \Gamma\to \bC^*: g(\gamma_1+\gamma_2)=(-1)^{\<\gamma_1,\gamma_2\>} g(\gamma_1)\cdot g(\gamma_2)\big\}.\]
It is a torsor for the algebraic torus $\Hom_\bZ(\Gamma,\bC^*)$, and hence non-canonically isomorphic to $(\bC^*)^4$. For each class $\gamma\in \Gamma$ there is a twisted character \[x_\gamma\colon \bT\to \bC^*, \qquad x_\gamma(g)=g(\gamma).\] 
The ray diagram associated to a BPS structure  consists of the rays $\bR_{>0}\cdot Z(\gamma)$ determined by those classes $\gamma\in \Gamma$ for which $\Omega(\gamma)\neq 0$. These rays are said to be active. The ray diagram for the BPS structure corresponding to a point $(v,w)\in M$ is illustrated in Figure \ref{fig} below. Note that the doubling procedure does not affect this.

The Riemann-Hilbert problem associated to the doubled BPS structure defined by a point of $T^*M$ depends also on the choice of an element $\xi\in \bT$ called the constant term. The problem then asks for a piecewise holomorphic map $\Phi\colon \bC^*\to \bT$
 which is holomorphic in the complement of the active rays, has a prescribed discontinuity as $t\in \bC^*$ crosses an active ray $\ell\subset \bC^*$, and has certain given limiting behaviour as $t\to 0$ or $t\to \infty$.  Composing with the twisted characters of $\bT$ we can equivalently encode the solution  in the system of maps \[\Phi_\gamma\colon \bC^*\to \bC^*, \qquad  \Phi_\gamma(t)=x_\gamma(\Phi(t))\]
 indexed by $\gamma\in \Gamma$.
 
 In a bit more detail, the required discontinuity as $t\in \bC^*$ crosses an active ray $\ell\subset \bC^*$ is 
\[\Phi_\beta(t)\mapsto \Phi_\beta(t)\cdot \prod_{Z(\gamma)\in \ell}(1-\Phi_\gamma(t))^{\,\Omega(\gamma)\<\gamma,\beta\>},\]
and we ask that 
\[\exp(Z(\gamma)/t)\cdot \Phi_\gamma(t)\to  x_\gamma(\xi),\]
as $t\to 0$, and that each $\Phi_\gamma(t)$ should have 
moderate growth as $t\to \infty$, in the sense that  there exists $k>0$ such that for all $|t|\gg 0$
\[|t|^{-k}<|\Phi_\gamma(t)|<|t|^k.\]
We review the precise details of the Riemann-Hilbert  problem in Sections \ref{two} and \ref{bps}.
Our first main result  can be summarised as follows.

 \begin{thm}
 \label{mememe}
 Consider the doubled BPS structure corresponding to a point of $T^*M$, and choose a constant term $\xi\in \bT$ which satisfies $x_\gamma(\xi)=1$ for all classes $\gamma\in \Gamma_{\leq 1}$. Then the corresponding Riemann-Hilbert problem  has a unique solution, which can be written explicitly in terms of Barnes double and triple sine functions.
 \end{thm}
 
We will give a  more precise statement of this result  in Section \ref{solution}, after the relevant special functions have been introduced in Section \ref{special}.

  
 \subsection{The $\tau$-function}
It turns out that the unique solutions of Theorem \ref{mememe} can be encoded in a single piecewise-holomorphic function $\tau=\tau(v,w,t)$. To do this we first re-express the unique solutions of Theorem \ref{mememe} in terms of  maps $\Psi_\gamma\colon \bC^*\to \bC^*$ by writing
 \[\exp(Z(\gamma)/t)\cdot \Phi_\gamma(t)=\Psi_\gamma(t)\cdot x_\gamma(\xi).\]
  It is easy to see that the maps $\Psi_\gamma$ are independent of the extended part of the central charge $Z_{\geq 2}$, and  therefore only depend on $v,w$ and $t$. We then look for a piecewise-holomorphic function $\tau=\tau(v,w,t)$ which is  invariant under simultaneous rescaling of all variables, and which satisfies
 \[\frac{\partial}{\partial t} \log  \Psi_{\beta^\vee} (t) =\frac{\partial}{\partial v} \log \tau(v,w,t), \qquad \frac{\partial}{\partial t} \log  \Psi_{\delta^\vee} (t)=  \frac{\partial}{\partial w} \log \tau(v,w,t)=0.\]
When it exists, such a function $\tau$ is easily seen to be unique up to  multiplication by a nonzero constant. We review the details of this definition in Section \ref{two}.

 In the case of the Riemann-Hilbert problems associated to the resolved conifold we  show that a $\tau$-function in the above sense does indeed exist, and we compute it explicitly.
 Let us introduce a function $K(v,w,t)$ via the integral representation 
 \begin{equation}\label{alba}K(v, w,t)=\exp\Bigg(-\int_C \frac{e^{vs}-1}{e^{ws}-1}\cdot \frac{e^{ts}}{(e^{ts}-1)^2}\cdot \frac{ds}{s}\Bigg),\end{equation}
 where the contour $C$ is the real axis with a small detour above the origin. This representation is
 valid for $0\leq \Re(v)\leq \Re(w)$. Let us also introduce
 \[R(v,w, t)= \Big(\frac{w}{2\pi it}\Big)^2 \big(\Li_3(e^{2\pi i v/w})-\zeta(3)\big)+\frac{i\pi}{12}\cdot\frac{ v}{w}.\]
 In Section \ref{solution} we shall prove 
 
 \begin{thm} \label{corr}The expression
 \[\tau(v,w,t)= K(v,w,t)\cdot \exp({R(v,w,t)})\]
 defines a $\tau$-function for the  variation of BPS structures defined by the resolved conifold.
As $t\to 0$ there is an asymptotic expansion
 \[ \log \tau(v,w,t)\,\sim\,  -\frac{1}{12}\log \Big(\frac{-w}{t}\Big) +\frac{i\pi}{12}\cdot \frac{v}{w} \]\[+
\sum_{g\geq 1}   \frac{B_{2g} \cdot \Li_{3-2g}(e^{2\pi i v/w})}{2g\cdot (2g-2)!  }\, \Bigg(\frac{2\pi i t}{w}\bigg)^{2g-2}+ \sum_{g\geq 2}\frac{ B_{2g} \cdot B_{2g-2}  }{2g\cdot (2g-2)\cdot  (2g-2)! }\, \bigg(\frac{2\pi i t}{w}\bigg)^{2g-2}
.\]
\end{thm}

The  positive degree part of the above series  reproduces the  free energy of the resolved conifold, with $2\pi t/w$ playing the role of the string coupling. In mathematical terms it is the generating function for the Gromov-Witten invariants of $X$.  Since the function $\tau$ has pleasant analytic properties it can be considered as a good candidate for a non-perturbative partition function of the conifold. There is quite a large theoretical physics literature on such non-perturbative partition functions, which the author is unfortunately not  competent to summarise. We merely note here that the expression \eqref{alba}   does indeed appear in the string theory literature: see for example equation (3.9) in \cite{KM} (with $\beta=1$). For more on non-perturbative partition functions in this context the reader    could start by consulting \cite{HO,Ko}.

\begin{remark}In \cite[Section 5]{RHDT} it is explained that for families of finite, integral,  uncoupled BPS structures (see \cite[Section 1]{RHDT} for precise definitions) the $\tau$-function is a finite product of Barnes $G$-functions, one for each nonzero BPS invariant. The BPS structures arising from curve-counting on Calabi-Yau threefolds are uncoupled and conjecturally integral, but they are certainly not finite.  Nonetheless, in \cite[Section 6]{RHDT}, the asymptotic expansion of the corresponding infinite formal product of Barnes $G$-functions is shown to reproduce that part of the topological string partition function arising from degenerate contributions of genus 0 curves, at least in postive degrees in the string coupling.  To produce a genuine analytic solution of the Riemann-Hilbert problem  in this way however, one has to make sense of a (presumably divergent)  infinite product of $G$-functions. The  function $K(v,w,t)$ appearing above can be thought of as regularisation of this infinite product. \end{remark}

\subsection{Some motivation from mirror symmetry}
\label{mot}

Classical mirror symmetry \cite{COGP,Giv} relates the periods of a Calabi-Yau threefold $Y$ to generating functions for enumerative invariants of a mirror Calabi-Yau threefold $X$. This can be viewed as an identification between two variations of Hodge structures (VHS). On one side is the classical VHS on the moduli space of complex structures on $Y$, considered in a neighbourhood of a  maximally unipotent degeneration. On the other is a VHS over the complexified K{\"a}her cone of $X$, constructed from the genus 0 Gromov-Witten (GW) invariants of $X$.

There is an obvious asymmetry here, in  that the moduli space of complex structures on $Y$ is a global space with interesting topology, which we are choosing to view near a given boundary point, whereas on the 
other side, the complexified K{\"a}her cone of $X$ has no interesting topology. To remove this asymmetry, we would like to see  the complexified K{\"a}hler cone as  (the universal cover of) a punctured neighbourhood of a boundary point in a larger space with non-trivial topology. Moreover one would like to be able to extend the VHS on the K{\"a}hler cone to a global VHS on this  `stringy' K{\"a}hler moduli space. Mirror symmetry should then give rise to a map between the complex moduli space of $X$ and the stringy K{\"a}hler moduli space of $Y$, identifying the two VHS.

Our best hope for a mathematical definition of such a stringy K{\"a}hler moduli space is via the space of stability conditions \cite{Br1,Do} on the derived category of coherent sheaves $\D(X)=\D^b\Coh(X)$. In fact this space $\Stab(X)$ is better thought of as being mirror to the space of deformations of the category $\D(Y)$, which contains the classical moduli space of complex structures on $Y$ within it (see \cite[Section 7]{Br2}). On this larger space one should expect a generalization of the notion of a VHS, which is referred to in \cite{KKP} as a non-commutative VHS. Nonetheless, the general conclusion remains the same: we should seek a geometric structure on the space of stability conditions which reproduces the A-model VHS in the large volume limit.

The definition of the GW invariants of $X$ is too geometric in nature to generalise to the derived category $\D(X)$. Instead, the natural enumerative invariants associated to points of the space $\Stab(X)$ are (generalized) DT invariants \cite{JS,KS1}, which encode the virtual Euler characteristics of moduli spaces of stable objects of each given Chern character. Although rank one DT invariants  are known to encode equivalent data to the GW invariants \cite{MNOP,PP}, the two systems of invariants have very different formal properties. In particular, it is not at all clear how to extract a VHS from DT theory; in the case of GW theory this arises from the geometric properties  of the compactification  of the moduli  of stable maps. 

The most  fundamental property of DT invariants is the Kontsevich-Soibelman wall-crossing formula \cite{KS1}. The strong analogy with the iso-Stokes condition for families of irregular connections  \cite{BTL} then suggests that the Riemann-Hilbert problem considered in \cite{RHDT}  might be the key to defining the required geometric structures on $\Stab(X)$. A closely-related version of this Riemann-Hilbert problem also plays a fundamental role in the work of Gaiotto, Moore and Neitzke \cite{GMN1}.
We view the calculations of this paper as an indication that this approach is on the right lines. We are using DT invariants as the basis for a non-perturbative construction  on the space of stability  conditions, which is appropriately invariant under the group of autoequivalences, and which  reproduces Gromov-Witten theory near the large volume limit.

\subsection*{Acknowledgements} I thank   Alba Grassi,  Kohei Iwaki and Bal{\' a}zs Szendr{\H o}i for useful remarks. I am particularly grateful to Simon Ruijsenaars for his expert help with  multiple sine functions. I would also like to thank the anonymous referee for his or her careful reading of the manuscript.

\section{BPS structures and Riemann-Hilbert problems}
\label{two}

In this section we recall some definitions and results from \cite{RHDT}.

\subsection{BPS structures and their doubles}

We start with the following definition, which  abstracts the output of unrefined DT theory.

\begin{defn}
A BPS structure $(\Gamma,Z,\Omega)$  consists of \begin{itemize}
\item[(a)] A finite-rank free abelian group $\Gamma\isom \bZ^{\oplus n}$, equipped with a skew-symmetric form \[\<-,-\>\colon \Gamma \times \Gamma \to \bZ,\]

\item[(b)] A homomorphism of abelian groups
$Z\colon \Gamma\to \bC$,\smallskip

\item[(c)] A map of sets
$\Omega\colon \Gamma\to \bQ,$
\end{itemize}

satisfying the following properties:

\begin{itemize}
\item[(i)] Symmetry: $\Omega(-\gamma)=\Omega(\gamma)$ for all $\gamma\in \Gamma$,\smallskip
\item[(ii)] Support property: fixing a norm  $\|\cdot\|$ on the finite-dimensional vector space $\Gamma\tensor_\bZ \bR$, there is a constant $C>0$ such that 
\begin{equation}
\label{support}\Omega(\gamma)\neq 0 \implies |Z(\gamma)|> C\cdot \|\gamma\|.\end{equation}
\end{itemize}
\end{defn}
The group homomorphism $Z$ is  called the central charge. The rational numbers $\Omega(\gamma)$ are called BPS invariants.
The Donaldson-Thomas (DT) invariants  of a BPS structure $(\Gamma,Z,\Omega)$ are defined by the expression
\begin{equation}
\label{bps1}\DT(\gamma)=\sum_{ \gamma=m\alpha}   \frac{1}{m^2} \, \Omega(\alpha)\in \bQ,\end{equation}
where the sum is over integers $m>0$ such that $\gamma$ is divisible by $m$ in the lattice $\Gamma$. A class $\gamma\in \Gamma$ is called active if $\Omega(\gamma)\neq 0$. 

A BPS structure $(Z,\Gamma,\Omega)$ will be called
\begin{itemize}
\item[(i)] 
 convergent, if for some $R>0$
\begin{equation}\label{concon}\big.\sum_{\gamma\in \Gamma} |\Omega(\gamma)|\cdot e^{-R|Z(\gamma)|}<\infty,\end{equation}
\item[(ii)] {uncoupled}, if  $\<\gamma_1,\gamma_2\>=0$ for any two active classes $\gamma_1,\gamma_2\in \Gamma$.
 \end{itemize}

Given 
a BPS structure $(\Gamma, Z, \Omega)$ the doubled BPS structure takes the form
\[(\Gamma\oplus \Gamma^\dual, Z\oplus Z^\dual, \Omega),\]
where $\Gamma^\dual=\Hom_\bZ(\Gamma,\bZ)$ is the dual lattice, and $Z^\dual\colon \Gamma^\dual\to \bC$ is an arbitrary group homomorphism. We equip the doubled lattice \[\Gamma_D=\Gamma\oplus \Gamma^\dual\] with the non-degenerate skew-symmetric form
\begin{equation}
\label{job}\big\<(\gamma_1,\lambda_1), (\gamma_2,\lambda_2)\big\>=\<\gamma_1,\gamma_2\>+\lambda_1(\gamma_2)-\lambda_2(\gamma_1).\end{equation}
The central charge is defined by
\[Z(\gamma,\lambda)=Z(\gamma)+Z^\vee(\lambda),\]
and the BPS invariant $\Omega(\gamma,\lambda)$ is defined to be zero unless $\lambda=0$ in which case $\Omega(\gamma,0)=\Omega(\gamma)$. 

\subsection{Twisted torus}
Given 
a BPS structure $(\Gamma, Z, \Omega)$, we  consider 
the algebraic torus
\[\bT_+=\Hom_\bZ(\Gamma,\bC^*)\isom (\bC^*)^n,\]
and its co-ordinate ring (which is also the group ring of the lattice $\Gamma$)
\[\bC[\bT_+]=\bC[\Gamma]\isom \bC[y_1^{\pm 1}, \cdots, y_n^{\pm n}].\]
We write $y_\gamma\in \bC[\bT_+]$ for the  character of $\bT_+$ corresponding to an element $\gamma\in \Gamma$. The skew-symmetric form $\<-,-\>$ induces an invariant Poisson structure on $\bT_+$, given on characters by
\begin{equation}
\label{poisson}\{y_{\alpha}, y_{\beta}\}= \<\alpha,\beta\>\cdot y_{\alpha}\cdot y_{\beta}.\end{equation}
The twisted torus of the BPS structure is a torsor over $\bT_+$ defined by
\[\bT= \bT_-=\{g\colon \Gamma \to \bC^*: g(\gamma_1+\gamma_2)=(-1)^{\<\gamma_1,\gamma_2\>} g(\gamma_1)\cdot g(\gamma_2)\},\]
The co-ordinate ring of $\bT_-$ is spanned as a vector space by the twisted characters
$x_\gamma\colon \bT_-\to \bC^*$ tautologically defined by $x_\gamma(g)=g(\gamma)\in \bC^*$.  Thus
\begin{equation}\bC[\bT_-]=\bigoplus_{\gamma\in \Gamma} \bC\cdot x_\gamma, \qquad x_{\gamma_1}\cdot x_{\gamma_2}=(-1)^{\<\gamma_1,\gamma_2\>}\cdot x_{\gamma_1+\gamma_2}.\end{equation}
There is a Poisson bracket on $\bC[\bT_-]$ given on twisted characters by 
 \begin{equation}
\label{poisson2}\{x_{\alpha}, x_{\beta}\}= \<\alpha,\beta\>\cdot x_{\alpha}\cdot x_{\beta}.\end{equation}

Associated to any ray $\ell=\bR_{>0}\cdot z\subset \bC^*$ is a formal sum of twisted characters
\begin{equation}
\label{hamiltonian}\DT(\ell)=-\hspace{-.8em}\sum_{\gamma\in \Gamma: Z(\gamma)\in \ell} \DT(\gamma) \cdot x_\gamma.\end{equation}
The ray is called active if this sum is nonzero, that is, if it contains a point $Z(\gamma)$ for some active class $\gamma\in \Gamma$.
We would like to associate an automorphism $\bS(\ell)$ of the twisted torus $\bT$ to each active ray $\ell\subset \bC^*$ by taking the time 1 Hamiltonian flow of the function $\DT(\ell)$. In fact, as explained in the next subsection, the best we can hope for in general is a partially-defined automorphism of $\bT$.

\subsection{BPS automorphisms}

Let $(\Gamma, Z, \Omega)$ be a  convergent BPS structure and fix an acute sector $\Delta\subset \bC^*$. For each real number $R>0$ we define  $U_\Delta(R)\subset \bT$ to be the interior of the subset
\[\big\{g\in \bT: Z(\gamma)\in \Delta \text{ and } \Omega(\gamma)\neq 0\implies |g(\gamma)|<\exp(-R\|\gamma\|)\big\}\subset \bT.\]
It is proved in \cite[Appendix B]{RHDT} that this  is a non-empty open subset.
The height of an active ray $\ell\subset \bC^*$ is defined to be \[\height(\ell)=\inf\big\{|Z(\gamma)|: \gamma \in \Gamma\text{ such that } Z(\gamma)\in \ell\text{ and } \Omega(\gamma)\neq 0\big\}.\]Non-active rays are considered to have infinite height. The support property  ensures that for any $H>0$ there are only finitely many rays of height $<H$.
The following statement can be found in \cite[Section 4]{RHDT} and is proved in \cite[Appendix B]{RHDT}.
 
\begin{prop}
\label{propy}
For sufficiently large $R>0$ the following statements hold:
\begin{itemize}
\item[(i)] For each ray $\ell\subset \Delta$, the power series $\DT(\ell)$ is absolutely convergent on $U_\Delta(R)$, and hence defines a holomorphic function \[\DT(\ell)\colon U_\Delta(R)\to \bC.\]
\item[(ii)] The time 1 Hamiltonian flow of the function $\DT(\ell)$  with respect to the Poisson structure $\{-,-\}$ on $\bT$ defines a holomorphic  embedding
\[\bS(\ell)\colon U_\Delta(R)\to \bT,\]
which we view as a partially-defined automorphism of $\bT$.
\item[(iii)] For each $H>0$, the composition in clockwise order
\[\bS_{<H}(\Delta) = \bS_{\ell_1} \circ \bS_{\ell_2} \circ \cdots \circ \bS_{\ell_k},\]
 corresponding to the finitely many rays $\ell_i\subset \Delta$ of height $< H$ is well-defined on $U_\Delta(R)$, as is the pointwise limit \[\bS(\Delta)=\lim_{H\to \infty} \bS_{<H}(\Delta).\] 
\end{itemize}
\end{prop}

The following result is proved in \cite[Appendix B]{RHDT}.

\begin{prop}
\label{birat}
Fix a ray $\ell\subset \bC^*$ and suppose that the following conditions are satisfied:
\begin{itemize}
\item[(i)] there are only finitely many active classes $\gamma\in \Gamma$ with $Z(\gamma)\in \ell$.\smallskip

\item[(ii)] any two active classes $\gamma_i\in \Gamma$ with $Z(\gamma_i)\in \ell$ satisfy $\<\gamma_i,\gamma_j\>=0$,\smallskip

\item[(iii)]  any active class $\gamma\in \Gamma$ with $Z(\gamma)\in \ell$ has $\Omega(\gamma)\in \bZ$.
\end{itemize}
Then the partially-defined automorphism $\bS(\ell)$ of Proposition \ref{propy} extends to a birational automorphism of $\bT$ whose pullback on twisted characters is given by
\begin{equation}\bS(\ell)^*(x_\beta)=x_\beta\cdot \prod_{Z(\gamma)\in \ell}(1-x_\gamma)^{\,\Omega(\gamma)\<\gamma,\beta\>}.\end{equation}
\end{prop}

\subsection{Riemann-Hilbert problem}
Let $(\Gamma, Z, \Omega)$ be a convergent BPS structure. Given a ray $\ell\subset \bC^*$ we consider the corresponding half-plane 
\[\bH_\ell=\ell\cdot \fh=\{z\in \bC^*:z=u\cdot v \text{ with } u\in \ell\text{ and }\Re(v)>0\},\]
 centered on it.
We shall be dealing with functions  $\Phi_\ell\colon \bH_\ell\to  \bT.$ Composing with the twisted characters of $\bT$ we can equivalently consider functions
\[\Phi_{\ell,\gamma}\colon \bH_\ell\to \bC^*,\qquad \Phi_{\ell,\gamma}(t)=x_\gamma(\Phi_\ell(t)).\]
The Riemann-Hilbert problem associated to the BPS structure $(\Gamma, Z, \Omega)$ depends on a choice of element $\xi\in \bT$ which we refer to as the constant term.

\begin{problem}
\label{dtsect}
Fix an element $\xi\in \bT$. For each non-active ray $\ell\subset \bC^*$ we  seek a holomorphic function
$\Phi_\ell\colon \bH_\ell\to \bT$
such that the following three conditions are satisfied:
\begin{itemize}

\item[(RH1)] {\it Jumping.} Suppose that two  non-active rays $\ell_1,\ell_2\subset \bC^*$ form the  boundary rays of an acute sector $\Delta\subset \bC^*$ taken in clockwise order. Then
\[\Phi_{\ell_2}(t)= \bS(\Delta) \circ \Phi_{\ell_1}(t), \]
for all $t\in \bH_{\ell_1}\cap \bH_{\ell_2}$ with  $ 0<|t|\ll 1$.
\smallskip

\item[(RH2)] {\it Finite limit at $0$}.
For each non-active ray $\ell\subset \bC^*$ and each class $\gamma\in \Gamma$ we have
\[\exp(Z(\gamma)/t)\cdot \Phi_{\ell,\gamma}(t) \to \xi(\gamma)\]
as $t\to 0$ in the half-plane $\bH_\ell$.\smallskip

\item[(RH3)] {\it Polynomial growth at $\infty$}. For any class $\gamma\in \Gamma$ and any non-active ray $\ell\subset \bC^*$, there exists $k>0$ such that
\[|t|^{-k} < |\Phi_{\ell,\gamma}
(t)|<|t|^k,\]
for  $t\in \bH_\ell$ satisfying $|t|\gg 0$.
\end{itemize}
\end{problem}

To make sense of the condition (RH1) note that condition (RH2) implies that  for any $R>0$ we have $\Phi_{\ell_i}(t)\in U_\Delta(R)$ whenever  $t\in \bH_{\ell_1}\cap \bH_{\ell_2}$ with  $ 0<|t|\ll 1$. But we can find  $R>0$ such that the partially-defined automorphism $\bS(\Delta)$ is well-defined on $U_{\Delta}(R)\subset \bT$.  Thus  the given relation does indeed make sense.

It will be useful to consider the maps $\Psi_\ell\colon \bH_\ell\to \bT_+$ defined by
\[\exp(Z/t)\cdot \Phi_\ell(t)=\Psi_\ell(t)\cdot \xi.\]
 Composing with the characters of $\bT_+$ we can also encode the solution in the system of maps
\[\Psi_{\ell,\gamma}(t)=x_\gamma(\Psi_\ell(t))=\exp(Z(\gamma)/t)\cdot \Phi_{\ell,\gamma}(t) \cdot \xi(\gamma)^{-1}.\]
Of course the maps $\Phi_\ell$ and $\Psi_\ell$ are equivalent data: we use whichever is most convenient.

An element $\gamma\in \Gamma$ will be called null if it satisfies $\<\alpha,\gamma\>=0$ for all active classes $\alpha\in \Gamma$. Note that the definition of the wall-crossing automorphisms $\bS(\ell)$ then implies that
$\bS(\ell)^*(x_\gamma)=x_\gamma$.
The following result is proved in  \cite[Section 4.5]{RHDT}.

\begin{prop}
\label{thisone}
Let $(\Gamma,Z,\Omega)$ be a convergent BPS structure, and fix a constant term $\xi\in \bT$.
\begin{itemize}
\item[(i)]If a class $\gamma\in \Gamma$ is null then any solution to Problem \ref{dtsect} satisfies
$\Psi_{\ell, \gamma}(t)=1$
for all $t\in \bH_\ell$.\smallskip

\item[(ii)] If $(\Gamma,Z,\Omega)$ is uncoupled then Problem \ref{dtsect} has at most one solution.
\end{itemize}
\end{prop}

\subsection{The $\tau$-function}

A variation of BPS structures $(\Gamma_p,Z_p,\Omega_p)$ over a complex manifold $M$ consists of a family of BPS structures indexed by the points $p\in M$ satisfying certain axioms, the most important of which is the Kontsevich-Soibelman wall-crossing formula, which describes the way  the BPS invariants $\Omega_p(\gamma)$ change as the point $p\in M$ moves. A complete definition is given in \cite[Appendix A]{RHDT}. 

For the purposes of this paper however, it is sufficient to consider the much simpler  notion of a  framed variations of uncoupled BPS structures over a complex manifold $M$, which is nothing more than a family of uncoupled BPS structures $(\Gamma,Z_p,\Omega)$ indexed by the points $p\in M$, such that the lattice $\Gamma$, the form $\<-,-\>$,  and the BPS invariants $\Omega(\gamma)$ are all constant, and such that for any $\gamma\in \Gamma$, the central charge $Z_p(\gamma)\in  \bC$ varies holomorphically. 

Given a variation of BPS structures, the obvious  map
\[\pi\colon M\to \Hom_\bZ(\Gamma,\bC)\isom \bC^n, \quad p\mapsto Z_p\]
is called the period map, 
and  the variation is called miniversal if it is  a local isomorphism. In that case, if we choose a basis $(\gamma_1,\cdots, \gamma_n)\subset \Gamma$, the functions $z_i=Z(\gamma_i)$ form a system of local co-ordinates near any given point of $M$.

For each point $p\in M$ we can consider the Riemann-Hilbert problem associated to the BPS structure $(\Gamma,Z_p,\Omega)$, and as $p\in M$ varies seek a family of solutions  given by a piecewise-holomorphic map
\[\Psi\colon M\times \bC^*\to \bT_+,\]
 which we view as a  function of the co-ordinates  $(z_1,\cdots, z_n)\in \bC^n$ and the parameter $t\in \bC^*$.  We define a  $\tau$-function for this  family of solutions to be a piecewise-holomorphic map \[\tau\colon M\times \bC^* \to \bC^*,\]
 which is invariant under simultaneous  rescaling of all co-ordinates $z_i$ and the parameter $t$, and which satisfies  the equations
 \begin{equation}
\label{hungry}\frac{1}{2\pi i}\cdot \frac{\partial \log \Psi_{\gamma_j}}{\partial t}=\sum_i \<\gamma_i,\gamma_j\> \frac{\partial \log \tau}{\partial z_i}. \end{equation}
When the form $\<-,-\>$ is non-degenerate these conditions  determine $\tau$ uniquely up to multiplication by a constant scalar factor. The author does not yet have a good explanation of why such a function $\tau$ should exist in general, but in the case of the variation of BPS structures associated to the resolved conifold, such a function can indeed be defined.


\section{The conifold  Riemann-Hilbert problem}
\label{bps}

In this section we give an explicit description of the family of Riemann-Hilbert problems associated to the resolved conifold.
The starting point is the  family of BPS structures arising from DT theory applied to the derived category of coherent sheaves on the resolved conifold. We recall the relevant results on the space of stability conditions and the  DT invariants in Appendix A, but it is not necessary to understand this material to follow the rest of the paper.

\subsection{The BPS structures}
\label{skip}

The BPS structures we shall consider depend on a point  in the space
\[M=\big\{(v,w)\in \bC^2: w\neq 0 \text{ and }v+nw\neq 0\text { for all }n\in \bZ\big\}\subset \bC^2.\]
Mathematically speaking, as we recall in Appendix A, this is the space
 of stability conditions on the derived category of the resolved conifold, quotiented by the subgroup of the autoequivalence group  generated by spherical twists. From the physical standpoint it can be thought of as the smallest unramified cover of the natural  $\bC^*$-bundle over the stringy K{\"a}hler moduli space on which the central charges of branes are single-valued.
We decompose \[M=M_+\sqcup M_0\sqcup M_-\]  according to the sign of $\Im(v/w)$.
The BPS structure $(\Gamma_{\leq 1},Z_{\leq 1},\Omega)$ corresponding to a point $(v,w)\in M$ is given by

\begin{itemize}
\item[(i)]
The lattice $\Gamma_{\leq 1}=\bZ\beta\oplus \bZ\delta$ with the form $\<-,-\>=0$.\smallskip

\item[(ii)]The central charge $Z_{\leq 1}\colon \Gamma_{\leq 1}\to \bC$ defined by
\[Z_{\leq 1}(a\beta+b\delta)=2\pi i(av+bw).\]

\item[(iii)]
The non-zero BPS invariants
\begin{equation}
\label{partiii}\Omega(\gamma)=\begin{cases} 1 &\text{if }\gamma=\pm \beta+n\delta \text{ with } n\in \bZ,\\
-2 &\text{if }\gamma=k\delta\text{ with }k\in \bZ\setminus\{0\}.\end{cases}\end{equation}
\end{itemize}

Together these structures form a framed and miniversal variation of uncoupled BPS structures over $M$. The BPS invariants are constant because  $\<-,-\>=0$. 
We also consider the corresponding doubled BPS structures. As in the introduction we denote these by $(\Gamma,Z,\Omega)$ and use the notation
\[\Gamma=\Gamma_{\leq 1} \oplus \Gamma_{\geq 2}, \qquad \Gamma_{\geq 2}:=\Gamma_{\leq 1}^\vee=\Hom_\bZ(\Gamma_{\leq 1},\bZ).\]
We denote by $(\beta^\vee,\delta^\vee)\subset \Gamma_{\geq 2}$ the dual basis to $(\beta,\delta)\subset \Gamma_{\leq 1}$. 
The map $\Omega\colon \Gamma\to \bQ$ satisfies $\Omega(\gamma)=0$ unless $\gamma\in \Gamma_{\leq 1}$. The central charge takes the form \[Z=Z_{\leq 1}\oplus Z_{\geq 2}\colon \Gamma \to \bC,\] where the group homomorphism $Z_{\geq 2}\colon \Gamma_{\geq 2}\to \bC$
is arbitrary. The resulting BPS  structures are all convergent, because
\[\sum_{\gamma\in \Gamma} |\Omega(\gamma)|\cdot e^{-|Z(\gamma)|}=2\sum_{n\in \bZ} e^{-|v+nw|} + 4\sum_{k>0} e^{-k|w|}<\infty.\]
Thus we obtain a  framed and miniversal variation of convergent, uncoupled  BPS structures  over the cotangent bundle $T^*M$.

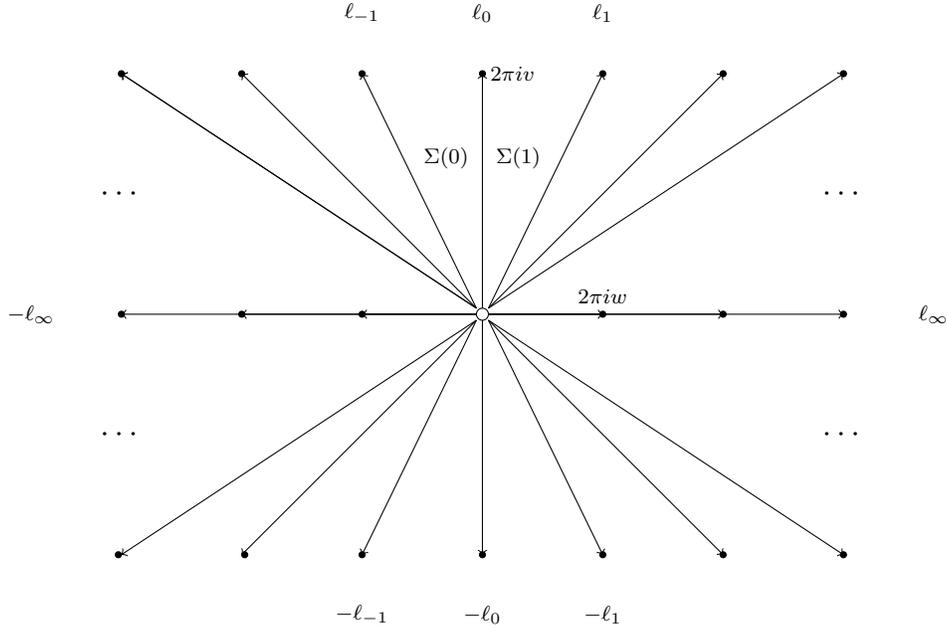
\begin{figure}
\begin{tikzpicture}
[scale=0.8]
\draw (5,0) circle [radius=0.1];
\draw[->] (4.9,0.1) -- (-1,4);
\draw[->] (4.9,0.1) -- (1,4);
\draw[->] (4.9,0.1) -- (3,4);
\draw[->] (5,0.1) -- (5,4);
\draw[->] (5.1,0.1) -- (7,4);
\draw[->] (5.1,0.1) -- (9,4);
\draw[->] (5.1,0.1) -- (11,4);

\draw[->] (4.9,0.1) -- (-1,4);

\draw[fill] (-1,4) circle [radius=0.05];
\draw[fill] (1,4) circle [radius=0.05];
\draw[fill] (3,4) circle [radius=0.05];
\draw[fill] (5,4) circle [radius=0.05];
\draw[fill] (7,4) circle [radius=0.05];
\draw[fill] (9,4) circle [radius=0.05];
\draw[fill] (11,4) circle [radius=0.05];


\draw[->] (4.9,-0.1) -- (-1,-4);
\draw[->] (4.9,-0.1) -- (1,-4);
\draw[->] (4.9,-0.1) -- (3,-4);
\draw[->] (5,-0.1) -- (5,-4);
\draw[->] (5.1,-0.1) -- (7,-4);
\draw[->] (5.1,-0.1) -- (9,-4);
\draw[->] (5.1,-0.1) -- (11,-4);

\draw[fill] (-1.05,-4) circle [radius=0.05];
\draw[fill] (1.05,-4) circle [radius=0.05];
\draw[fill] (3,-4) circle [radius=0.05];
\draw[fill] (5,-4) circle [radius=0.05];
\draw[fill] (7,-4) circle [radius=0.05];
\draw[fill] (9,-4) circle [radius=0.05];
\draw[fill] (11,-4) circle [radius=0.05];


\draw[->] (4.9,0) -- (3,0); \draw[->] (4.9,0) -- (1,0); \draw[->] (4.9,0) -- (-1,0);
\draw[->] (5.1,0) -- (7,0); \draw[->] (5.1,0) -- (9,0); \draw[->] (5.1,0) -- (11,0);

\draw[fill] (3,0) circle [radius=0.05];
\draw[fill] (1,0) circle [radius=0.05];
\draw[fill] (-1,0) circle [radius=0.05];
\draw[fill] (7,0) circle [radius=0.05];
\draw[fill] (9,0) circle [radius=0.05];
\draw[fill] (11,0) circle [radius=0.05];


\draw (11,2) node { $\cdots$};
\draw (-1,2) node { $\cdots$};
\draw (11,-2) node { $\cdots$};
\draw (-1,-2) node { $\cdots$};

\draw (5,5) node {$\scriptstyle \ell_0$};
\draw (3,5) node {$\scriptstyle \ell_{-1}$};
\draw (7,5) node {$\scriptstyle \ell_{1}$};
\draw (-2.5,0) node {$\scriptstyle -\ell_{\infty}$};
\draw (12.5,0) node {$\scriptstyle \ell_{\infty}$};

\draw (5,-5) node {$\scriptstyle -\ell_0$};
\draw (3,-5) node {$\scriptstyle -\ell_{-1}$};
\draw (7,-5) node {$\scriptstyle -\ell_{1}$};

\draw (4.4,2.6) node {$\scriptstyle \Sigma(0)$};
\draw (5.6,2.6) node {$\scriptstyle \Sigma(1)$};
\draw (7,0.3) node {$\scriptstyle 2\pi i w$};
\draw (5.5,4) node {$\scriptstyle 2\pi i v$};


\end{tikzpicture}
\caption{The ray diagram associated to a point $(v,w)\in M_+$.
\label{fig}}
\end{figure}

\subsection{BPS automorphisms}
Let us fix a point $(v,w)\in M_+$.  Define rays
\[\ell_\infty=\bR_{>0}\cdot 2\pi i w,\qquad \ell_n=\bR_{>0}\cdot 2\pi i (v+nw)\subset \bC^*.\]
The active rays for the  corresponding BPS structure  $(\Gamma,Z,\Omega)$ defined above  are precisely the rays $\pm \ell_\infty$ and $\pm \ell_n$ for $n\in \bZ$. The corresponding ray diagram is illustrated in Figure \ref{fig}. We let $\Sigma(n)$ be the convex open sector with boundary rays $\ell_{n-1}$ and $\ell_{n}$. 
Note that the union of the active rays is a closed subset of $\bC^*$ whose open complement is the disjoint union of the open sectors $\pm \Sigma(n)\subset \bC^*$.

We shall now describe explicitly the BPS automorphisms $\bS(\ell)$ of the twisted torus $\bT$ associated to the doubled lattice $\Gamma=\Gamma_{\leq 1}\oplus \Gamma_{\geq 2}$.  We denote by $x_\gamma\colon \bT\to \bC^*$ the twisted character corresponding to an element $\gamma\in \Gamma$. Since all active classes lie in $\Gamma_{\leq 1}\subset \Gamma$, and  the form $\<-,-\>$ is zero on $\Gamma_{\leq 1}$, it follows that all classes $\gamma\in \Gamma_{\leq 1}$ are null, and hence all BPS automorphisms act trivially on the corresponding twisted characters $x_\gamma$.

Proposition \ref{birat} shows that the BPS automorphism  associated to the ray $\ell_n$ takes the form
\[\bS(\ell_n)^*(x_\gamma)=x_\gamma\cdot (1-x_{\beta+n\delta})^{\<\beta+n\delta,\gamma\>}.\]
In particular, the twisted characters for the generators $(\beta^\vee,\delta^\vee)\subset \Gamma_{\geq 2}$ transform as
\begin{equation}
\label{da}\bS(\ell_n)^*(x_{\beta^\vee})=x_{\beta^\vee} \cdot 
 (1-x_{\beta+n\delta})^{-1}, \quad \bS(\ell_n)^*(x_{\delta^\vee})= x_{\delta^\vee}\cdot (1-x_{\beta+n\delta})^{-n}.\end{equation}
 
 Since the ray $\ell_\infty$ contains infinitely many active classes, Proposition \ref{birat} no longer applies. Nonetheless, Proposition \ref{propy} shows that $\bS(\ell_\infty)$ exists on a suitable open subset of $\bT$, and then the same calculation as  the proof of Proposition \ref{birat} (see \cite[Appendix B]{RHDT}) shows that its pullback on twisted characters is given by\[\bS(\ell_\infty)^*(x_\gamma)=x_\gamma\cdot \prod_{k\geq 1} (1-x_{k\delta})^{-2k\cdot \<\delta,\gamma\>}.\]
 It then follows that $\bS(\ell)$ extends to the analytic open subset of $\bT$ where $|x_\delta|<1$, but not to a Zariski open subset.
 The action on the basic twisted characters as above is
\[\bS(\ell_\infty)^*(x_{\beta^\vee})=x_{\beta^\vee}, \qquad \bS(\ell_\infty)^*(x_{\delta^\vee})= x_{\delta^\vee}\cdot \prod_{k\geq 1} (1-x_{k\delta})^{2k}.\]
   
\begin{remark}
\label{toro}There is no need to consider the active rays $-\ell_n$ and $-\ell_\infty$ separately, since as explained in \cite[Section 4.4]{RHDT}, for any ray $\ell\subset \bC^*$ there is a relation
\begin{equation}
\label{sym}\bS(-\ell)\circ \sigma=\sigma\circ \bS(\ell),\end{equation}
where $\sigma\colon \bT\to \bT$ is the involution  which acts on twisted characters as $x_\gamma\leftrightarrow x_{-\gamma}$.
\end{remark}

We shall also need to describe the BPS automorphisms $\bS(\Delta)$ associated to acute sectors $\Delta\subset \bC^*$. There are two possibilities:
either $\Delta$ contains a finite number of active rays, or it contains one of the two rays $\pm\ell_\infty$, and hence also an infinite number of the rays $\pm \ell_n$.
 In the first case the corresponding BPS automorphism $\bS(\Delta)$ is a finite composition of the birational automorphisms $\bS(\ell)$ and nothing more needs to be said.
For the second case, we can suppose by Remark \ref{toro} that  $\Delta$ contains the ray $\ell_\infty$. Since we understand finite compositions of the maps $\bS(\ell_n)$ it is enough to consider the extreme case when $\Delta$ is just less than a half-plane, so that its bounding rays lie in  sectors $\Sigma(m)$ and $-\Sigma(m)$, and without loss of generality we can take $m=0$.

The BPS automorphism $\bS(\Delta)$ is guaranteed to exist on some suitable open subset of $\bT$ by Proposition \ref{propy}. By definition it is the limit as $H\to \infty$ of the finite composition of BPS automorphisms corresponding to rays in $\Delta$ of height $< H$. 
Note that all the BPS automorphisms $\bS(\ell)$ commute  so there is no need to distinguish the order of these compositions.  Since the active rays contained in $\Sigma$ are
$\ell_n\text{ for } n\geq 0$,  $-\ell_{n}\text{ for } n<0$, and $ \ell_\infty$, it follows that $\bS(\Delta)$  satisfies
\begin{equation}\label{da2}\bS(\Delta)^*(x_\gamma)=x_\gamma \cdot \prod_{n\geq 0}  (1-x_{\beta+n\delta})^{\<\beta+n\delta,\gamma\>}\cdot \prod_{n\geq  1}(1-x_{-(\beta-n\delta)})^{-\<\beta-n\delta,\gamma\>} \cdot \prod_{k\geq 1}(1-x_{k\delta})^{-2k\cdot \<\delta,\gamma\>}\end{equation}
Once again, it follows that $\bS(\Delta)$ is well-defined on the analytic open subset $|x_\delta|<1$.

\subsection{The Riemann-Hilbert problem}

We now consider the Riemann-Hilbert problem defined by the doubled BPS structure $(\Gamma,Z,\Omega)$ corresponding to a point of $T^*M$, together with a fixed choice of constant term $\xi=(\xi_{\leq _1},\xi_{\geq 2})\in \bT$. Since these structures are uncoupled and convergent, Proposition \ref{thisone} ensures that there is at most one solution. We shall always assume that our constant term satisfies $\xi_{\leq 1}=1$, that is  that
$\xi(\gamma)=1$ for all $\gamma\in \Gamma_{\leq 1}$.
 We do not currently know how to solve the Riemann-Hilbert problem without this simplifying assumption.
  For now we shall also assume  that $(v,w)\in M_+$: for other cases see Section \ref{deg}.
 
 \begin{remark}
\label{canada}The symmetry \eqref{sym} implies that any solution  to the Riemann-Hilbert problem satisfies
 \begin{equation}
 \label{tor}\Phi^{\sigma(\xi)}_{-\ell,-\gamma}(-t)=\Phi_{\ell,\gamma}^{\xi}(t).\end{equation}
 Indeed, this follows from the observation of \cite[Section 4.4]{RHDT} once one has the uniqueness result of Proposition \ref{thisone}.
 \end{remark}
 
 Given the assumption $\xi_{\leq 1}=1$, our Riemann-Hilbert problem depends on the point $(v,w)\in M_+$,  together with the extra data of  homomorphisms  \begin{equation}
 \label{choices}Z^\vee\colon \Gamma_{\geq 1}\to \bC, \qquad \xi^\vee\colon \Gamma_{\geq 1}\to \bC^*.\end{equation}
The solution $\Phi_\ell\colon \bH_\ell\to \bT$ does not depend in a very interesting way on this extra data. In fact it is easy to see that  the maps $\Psi_\ell\colon \bH_\ell\to \bT_+$ defined by
\[\exp(Z/t)\cdot \Phi_\ell(t)=\Psi_\ell(t)\cdot \xi\]
are independent of $(Z^\vee,\xi^\vee)$. We shall therefore make the trivial choice $Z^\vee=0$ and $\xi^\vee=1$.

 Since all  classes $\gamma\in \Gamma$ are null, Proposition \ref{thisone} shows that
\begin{equation}
\label{awe}\Phi_{\ell,\gamma}^\xi(t)=e^{-Z(\gamma)/t},\end{equation}
for all non-active rays $\ell\subset \bC^*$, and all classes $\gamma\in \Gamma_{\leq 1}$.
It follows that a solution to the Riemann-Hilbert  is  specified by the functions
\[B_n(t)=B_n(v,w,t)=\Phi_{r_n,\beta^\vee}(t), \qquad D_n(t)=D_n(v,w,t)= \Phi_{r_n,\delta^\vee}(t),\]
where $r_n\subset \Sigma(n)$ is an arbitrary non-active ray lying in the given sector.
There is no need to consider the functions $\Phi_{\ell,\gamma}(t)$ for non-active rays $\ell\subset \bC^*$  lying in the opposite sectors $-\Sigma(n)$ since these are taken care of by  Remarks \ref{toro} and \ref{canada}  above.
Define the half-plane
\[\cH(n)=\{z\in \bC^*:z=ab\text{ with }a\in \ell_n\text{ and }\Re(b)>0\},\]
centered on the ray $\ell_n$.
Working out the conditions imposed on the functions $B_n$ and $D_n$ we obtain the following explicit version of the Riemann-Hilbert problem for the doubled BPS structure.

\begin{problem}
\label{tra}
Fix $(v,w)\in M_+$. For each $n\in \bZ$ find holomorphic functions $B_n(t)$ and $D_n(t)$  on the region\[\cV(n)=\cH(n-1)\cup\cH(n),\]
satisfying the following properties.
\begin{itemize}
\item[(i)] As $t\to 0$ in any closed subsector of $\cV(n)$ one has \[ B_n(t)\to  1, \qquad  D_n(t)\to  1.\]
\item[(ii)] For each $n\in \bZ$ there exists $k>0$ such that for  any closed subsector of $ \cV(n)$
\[ |t| ^{-k} <  |B_n(t)|,  |D_n(t)| < |t|^k, \qquad  |t| \gg 0.\]
\item[(iii)] On the intersection $\cH(n)=\cV(n)\cap \cV(n+1)$ there are relations
\[B_{n+1}(t)=B_{n}(t)\cdot (1-x q^{n} )^{-1}, \quad D_{n+1}(t)=D_{n}(t)\cdot (1-x q^{n})^{-n}.\]

\item[(iv)] Note that $\cV(0)\cap -\cV(0)=i \cdot \Sigma(0)\sqcup -i \cdot \Sigma(0)$. In the region $-i\cdot \Sigma(0)$  there are relations
\[B_0(t)\cdot B_0(-t)=\prod_{n\geq 0}\big(1-x q^{n}\big)\cdot \prod_{n\geq 1}\big(1-x^{-1} q^{n})^{-1},\]\begin{equation*} D_0(t)\cdot D_0(-t)=\prod_{n\geq 0}\big(1-xq^{n}\big)^{n}\cdot \prod_{n\geq 1}\big(1-x^{-1}  q^{n}\big)^{n}\cdot \prod_{k\geq 1} \big(1-q^{k}\big)^{-2k},\end{equation*}
\end{itemize}
where we used the notation \begin{equation}
\label{notat}x=\exp(-2\pi iv/ t), \qquad q=\exp(- 2\pi i w/t).\end{equation}
\end{problem}

Parts (iii) and (iv) arise from condition (RH1) of Problem \ref{dtsect}.
Part (iii) is obtained by plugging \eqref{awe} into \eqref{da}, and noting that the sectors $\Sigma(n)$ and $\Sigma(n+1)$ come in clockwise order. Similarly (iv) is obtained by plugging \eqref{awe} into \eqref{da2} and using \eqref{tor}.
Note that the last factor in the second equation of (iv) is the sole contribution of the  ray $\ell_\infty$. 

\begin{remark}
At first sight one might expect a simple solution to Problem \ref{tra}  in which
\[B_n(v,w,t)=\prod_{m\geq n}(1-xq^m).\]Although this function does indeed satisfy the relevant identity from Problem \ref{tra} (iii), it is not holomorphic, or even meromorphic, in the required half-plane. For example, $B_0(v,w,t)$, which is essentially the (exponential of) the quantum dilogarithm function, is ill-defined for $w/t\in \mathbb{Q}$, and thus fails to be holomorphic on a dense subset of the rays $\pm i\cdot \ell_\infty$.
\end{remark}

\subsection{Difference equations}
\label{diffrence}
Our variation of BPS structures carries  a 
 free action of $\bZ$. This symmetry will allow us to restate the above Riemann-Hilbert problem as a  pair of coupled difference equations.  Consider the action of $\bZ$ on the lattice $\Gamma$, preserving the form $\<-,-\>$, in which $m\in \bZ$ acts via
\[(\beta,\delta)\mapsto (\beta-m\delta, \delta),\quad (\beta^\vee,\delta^\vee)\mapsto (\beta^\vee,\delta^\vee+m\beta^\vee).\]
This induces an action on $T^*M$ by
\[(v,w) \mapsto (v+mw,w), \quad (v^\vee, w^\vee)\mapsto (v^\vee, w^\vee-mv^\vee).\]
More precisely, the map $m\colon \Gamma\to \Gamma$ defines an isomorphism between the BPS structure at a point $Z\in T^*M$, and the BPS structure at the point $m\cdot Z$. Note that   a point $t\in \bC^*$ lies in the sector $\Sigma(n)$ for the BPS structure defined by the point $(v,w)$ precisely if it lies in the sector $\Sigma(m+n)$ for the BPS structure defined by $(v-mw,w)$. A similar remark applies to the regions $\cH(n)$.

There is an obvious induced action on the constant terms $\xi\in \bT$, which preserves our choice $\xi=1$ and $\xi^\dual=1$. The uniqueness of solutions given by Proposition \ref{thisone} then implies that if we can solve Problem \ref{tra} for all $(v,w)\in M_+$  then the solution must satisfy
\begin{gather}\label{blob}B_n(v,w,t)= B_0(v+nw,w,t),\\ D_n(v,w,t)=D_0(v+nw,w,t) \cdot B_0(v+nw,w,t)^{n}.\end{gather}
We now restate Problem \ref{tra}  in terms of just two functions $B=B_0$ and $D=D_0$.

\begin{problem}
\label{tre}
Find holomorphic functions  $B(v,w,t)$ and $D(v,w,t)$  defined for $(v,w)\in M_+$ and $t\in \bC^*$ lying in the region
\[\cV(0)=\cH(-1)\cup\cH(0)\] satisfying the following properties:
\begin{itemize}
\item[(i)] For fixed $(v,w)\in M_+$, one has
\[B(v,w,t)\to 1, \qquad  D(v,w,t)\to  1,\]
as $t\to 0$ in any closed subsector of $\cV(0)$.\smallskip

\item[(ii)] For fixed $(v,w)\in M_+$ there  exists $k>0$  such that for any closed subsector of $ \cV(0)$
\[ |t| ^{-k} <  |B(v,w,t)|,  |D(v,w,t)| < |t|^k, \qquad  |t| \gg 0.\]

\item[(iii)] For $(v,w)\in M_+$ and $t\in \bC^*$ lying in the  intersection $\cH(0)=\cV(0)\cap\cV(1)$ 
there are relations
\[\frac{B(v+w,w,t)}{B(v,w,t)}=(1-x)^{-1}, \qquad \frac{D(v+w,w,t)}{D(v,w,t)}= B(v+w,w,t)^{-1}.\]

\item[(iv)] Note that $\cV(0)\cap -\cV(0)=i \cdot \Sigma(0)\sqcup -i \cdot \Sigma(0)$. In the region $-i\cdot \Sigma(0)$  there are relations
\[B(v,w,t)\cdot B(v,w,-t)=\prod_{n\geq 0}\big(1-x q^{n}\big)\cdot \prod_{n\geq 1}\big(1-x^{-1} q^{n}\big)^{-1},\]\begin{equation*} D(v,w,t)\cdot D(v,w,-t)=\prod_{n\geq 0}\big(1-xq^{n}\big)^{n}\cdot {\prod_{n\geq 1}\big(1-x^{-1}  q^{n}\big)^{n}}\cdot {\prod_{k\geq 1} \big(1-q^{k}\big)^{-2k}},\end{equation*}
\end{itemize}
where we used the notation \eqref{notat} as before.
\end{problem}

It is easy to see that a solution to Problem \ref{tre} gives rise to a solution to Problem \ref{tra} for all $(v,w)\in M_+$ via \eqref{blob}.
In particular, this implies  that Problem \ref{tre} has at most one solution.

\subsection{Symmetry and the degenerate case}
\label{deg}

So far we have considered the Riemann-Hilbert problems associated to points $(v,w)\in M_+$. We should now consider the problems associated to points in $M_-$ and $M_0$. 
 For this, note that there is an involution of $\Gamma$
\[(\beta,\delta)\mapsto (-\beta,\delta), \qquad (\beta^\vee,\delta^\vee)\mapsto (-\beta^\vee,\delta^\vee),\]
preserving the BPS invariants, which therefore identifies the BPS structure at a point $(v,w)\in M_+$ with the BPS structure at the corresponding point $(-v,w)\in M_-$. We need a new convention to label the BPS rays  for the BPS structures corresponding to points of $M_-$. We shall choose to label \[\ell_n=\bR_{>0} \cdot 2\pi i(-v+nw),\] with $\Sigma(n)$ lying between $\ell_{n}$ and $\ell_{n+1}$ as before. The result is illustrated in Figure \ref{fig3}. With these conventions the above involution identifies the ray diagrams for the BPS structures at $(v,w)$ and $(-v,w)$, so we  obtain relations
\[B_n(-v,w,t)=B_n(v,w,t)^{-1}, \qquad D_n(-v,w,t)=D_n(v,w,t).\]
Thus the solutions for points of $M_-$ are trivially related to those for $M_+$.

\begin{figure}
\begin{tikzpicture}
[scale=0.8]
\draw (5,0) circle [radius=0.1];
\draw[->] (4.9,0.1) -- (-1,4);
\draw[->] (4.9,0.1) -- (1,4);
\draw[->] (4.9,0.1) -- (3,4);
\draw[->] (5,0.1) -- (5,4);
\draw[->] (5.1,0.1) -- (7,4);
\draw[->] (5.1,0.1) -- (9,4);
\draw[->] (5.1,0.1) -- (11,4);

\draw[->] (4.9,0.1) -- (-1,4);

\draw[fill] (-1,4) circle [radius=0.05];
\draw[fill] (1,4) circle [radius=0.05];
\draw[fill] (3,4) circle [radius=0.05];
\draw[fill] (5,4) circle [radius=0.05];
\draw[fill] (7,4) circle [radius=0.05];
\draw[fill] (9,4) circle [radius=0.05];
\draw[fill] (11,4) circle [radius=0.05];


\draw[->] (4.9,-0.1) -- (-1,-4);
\draw[->] (4.9,-0.1) -- (1,-4);
\draw[->] (4.9,-0.1) -- (3,-4);
\draw[->] (5,-0.1) -- (5,-4);
\draw[->] (5.1,-0.1) -- (7,-4);
\draw[->] (5.1,-0.1) -- (9,-4);
\draw[->] (5.1,-0.1) -- (11,-4);

\draw[fill] (-1.05,-4) circle [radius=0.05];
\draw[fill] (1.05,-4) circle [radius=0.05];
\draw[fill] (3,-4) circle [radius=0.05];
\draw[fill] (5,-4) circle [radius=0.05];
\draw[fill] (7,-4) circle [radius=0.05];
\draw[fill] (9,-4) circle [radius=0.05];
\draw[fill] (11,-4) circle [radius=0.05];


\draw[->] (4.9,0) -- (3,0); \draw[->] (4.9,0) -- (1,0); \draw[->] (4.9,0) -- (-1,0);
\draw[->] (5.1,0) -- (7,0); \draw[->] (5.1,0) -- (9,0); \draw[->] (5.1,0) -- (11,0);

\draw[fill] (3,0) circle [radius=0.05];
\draw[fill] (1,0) circle [radius=0.05];
\draw[fill] (-1,0) circle [radius=0.05];
\draw[fill] (7,0) circle [radius=0.05];
\draw[fill] (9,0) circle [radius=0.05];
\draw[fill] (11,0) circle [radius=0.05];


\draw (11,2) node { $\cdots$};
\draw (-1,2) node { $\cdots$};
\draw (11,-2) node { $\cdots$};
\draw (-1,-2) node { $\cdots$};

\draw (5,5) node {$\scriptstyle \ell_0$};
\draw (3,5) node {$\scriptstyle \ell_{-1}$};
\draw (7,5) node {$\scriptstyle \ell_{1}$};
\draw (-2.5,0) node {$\scriptstyle -\ell_{\infty}$};
\draw (12.5,0) node {$\scriptstyle \ell_{\infty}$};

\draw (5,-5) node {$\scriptstyle -\ell_0$};
\draw (3,-5) node {$\scriptstyle -\ell_{-1}$};
\draw (7,-5) node {$\scriptstyle -\ell_{1}$};

\draw (4.4,2.6) node {$\scriptstyle \Sigma(0)$};
\draw (5.6,2.6) node {$\scriptstyle \Sigma(1)$};
\draw (7,0.3) node {$\scriptstyle 2\pi i w$};
\draw (5.5,-4) node {$\scriptstyle 2\pi i v$};


\end{tikzpicture}
\caption{The ray diagram associated to a point $(v,w)\in M_-.$
\label{fig3}}
\end{figure}
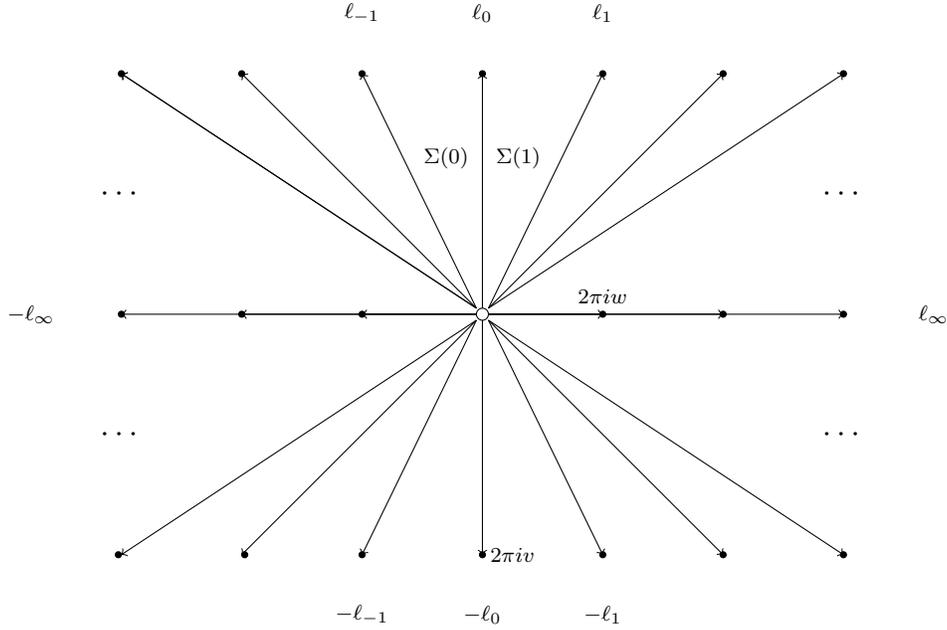

We also consider the degenerate case when $(v,w)\in M_0$. By applying the $\bZ$-action we can reduce to the case when $v/w\in (0,1)$. We set $\ell=\bR_{>0}\cdot 2\pi i w$  and define
\[\cH=\{z\in \bC^*:z=ab\text{ with }a\in \ell\text{ and }\Re(b)>0\},\]
 The only active rays are $\pm \ell$, and the wall-crossing formula implies that  the (partially-defined) BPS automorphism $\bS(\ell)$ coincides with the map $\bS(\Sigma)$ considered above.  The Riemann-Hilbert problem is then

\begin{problem}
\label{tro}
Fix $(v,w)\in M_0$ with $v/w\in (0,1)$. Find holomorphic functions  $B(t)$ and $D(t)$  on the region
\[\cV=\bC^*\setminus (\bR_{>0}\cdot w)\] satisfying the following properties:
\begin{itemize}
\item[(i)] As $t\to 0$ in any closed subsector of $\cV$ one has
\[ B(t)\to  1, \qquad   D(t)\to  1.\]
\item[(ii)] There  exists $k>0$ such that for any closed subsector of $ \cV$
\[ |t| ^{-k} <  |B(t)|,  |D(t)| < |t|^k, \qquad  |t| \gg 0.\]
\item[(iii)]  Note that $\cV\cap-\cV=\cH\sqcup-\cH$. In the region $\cH$  there are relations
\[B(t)\cdot B(-t)=\prod_{n\geq 0}\big(1-x q^{n}\big)\cdot \prod_{n\geq 1}\big(1-x^{-1} q^{n}\big)^{-1},\]\[ D(t)\cdot D(-t)=\prod_{n\geq 0}\big(1-xq^{n}\big)^{n}\cdot \prod_{n\geq 1}\big(1-x^{-1}  q^{n}\big)^{n}\cdot \prod_{k\geq 1} \big(1-q^{k}\big)^{-2k},\]
\end{itemize}
where we used the notation \eqref{notat} as before.
\end{problem}



\section{Double and triple sine functions}
\label{special}

In this section we introduce some special functions which we will later use  to solve the conifold Riemann-Hilbert problem described in the last section. The relevant special functions are, up to some exponential factors, the double sine function, and the triple sign function with two equal parameters. 

Multiple sine functions are usually defined using the multiple  gamma functions of Barnes \cite{Barnes4}. Both are functions of a variable $z\in \bC$ and $r$ parameters $\omega_1,\cdots,\omega_r\in \bC^*$. One has
\[\sin_r(z\b\omega_1,\cdots,\omega_r)=\Gamma_r(z\b\omega_1,\cdots,\omega_r)\cdot \Gamma_r\Big(\sum_{i=1}^r {\omega_i }-z\b\omega_1,\cdots,\omega_r\Big)^{(-1)^r}.\]
For definitions and results on multiple gamma and sine functions we recommend \cite{JM,KK,N,Ruj2}.

 \subsection{Double sine function}
 We begin by considering a function of $z\in \bC$ and two parameters $\omega_1,\omega_2\in \bC^*$. We shall use the notation
\begin{equation}
\label{bore}
 \begin{aligned}
  x_1=\exp(2\pi i z/\omega_1), &\quad x_2=\exp(2\pi i z/\omega_2), \\ q_1=\exp(2\pi i{\omega_2}/{\omega_1}), &\quad q_2=\exp(2\pi i{\omega_1}/{\omega_2}).\end{aligned}
 \end{equation}
Our function is obtained by multiplying the double sine function by an exponential prefactor. 
 The definition is
 \begin{equation}
 \label{when}F(z\b\omega_1,\omega_2)=e^{-\frac{\pi i}{2} \cdot  B_{2,2}(z\b\omega_1,\omega_2)}\cdot 
\sin_2(z\b \omega_1,\omega_2),\end{equation}
where $B_{2,2}(z\b\omega_1,\omega_2)$ is the multiple Bernoulli polynomial
\[B_{2,2}(z\b\omega_1,\omega_2)=\frac{z^2}{\omega_1\omega_2} -\Big(\frac{1}{\omega_1} + \frac{1}{\omega_2}\Big) z+\frac{1 }{6}\Big(\frac{\omega_2}{\omega_1} + \frac{\omega_1}{\omega_2}\Big)+\frac{1}{2}.\]
Up to trivial changes of variables the function $F(z\b\omega_1,\omega_2)$ coincides with the Fadeev dilogarithm appearing in the work of Fock and Goncharov on cluster theory \cite{FG}.  

 Although $F(z\b\omega_1,\omega_2)$  is a single-valued function of $z\in \bC$ for fixed values of $\omega_1,\omega_2\in \bC^*$, to make it a single-valued function of all three parameters we must introduce a cut-line. We will therefore only consider the function under the additional assumption that $\omega_1/\omega_2\notin \bR_{<0}$. 
  
 \begin{prop}
 \label{f}
The function $F(z\b\omega_1,\omega_2)$ is a single-valued meromorphic function of variables $z\in \bC$ and $\omega_1,\omega_2\in \bC^*$ under the assumption $\omega_1/\omega_2\notin \bR_{<0}$. It has the following properties:
\begin{itemize}
\item[(i)] The function is regular and non-vanishing except  at the points
 \[z=a\omega_1+b\omega_2, \quad a,b\in\bZ,\]
 which are   zeroes if $a,b\leq 0$, poles if $a,b>0$, and otherwise neither.\smallskip
 
\item[(ii)] It is symmetric in the arguments $\omega_1,\omega_2$:
 \[F(z\b\omega_1,\omega_2)=F(z\b\omega_2,\omega_1),\]
 and is invariant under simultaneous rescaling of all three arguments. \smallskip
 
\item[(iii)] It satisfies the two difference relations:
\begin{equation}
\label{diffy}\frac{F(z+\omega_1\b\omega_1,\omega_2)}{F(z\b\omega_1,\omega_2)}=\frac{1}{1-x_2}, \qquad \frac{F(z+\omega_2\b\omega_1,\omega_2)}{F(z\b\omega_1,\omega_2)}=
\frac{1}{1-x_1}.\end{equation}

\item[(iv)] There is a product expansion
 \[ F(z\b\omega_1,\omega_2)=\prod_{k\geq 0} (1-x_1 q_1^{-k})^{-1}\cdot \prod_{k\geq 1} (1-x_2 q_2^{k}),\]
 valid when  $\Im(\omega_1/\omega_2)>0$.\smallskip

 \item[(v)] When $\Re(\omega_i)>0$ and $0<\Re(z)<\Re(\omega_1+\omega_2)$ there is an integral representation \begin{equation}
 \label{homeagain}F(z\b\omega_1,\omega_2)=\exp\Bigg(\int_C \frac{e^{zs}}{(e^{\omega_1 s}-1)(e^{\omega_2 s}-1)}\frac{ds}{s}\Bigg),\end{equation}
 where the contour $C$ follows the real axis from $-\infty$ to $+\infty$ avoiding the origin by a small detour in the upper half-plane.
 \end{itemize}
 \end{prop}
 
\begin{proof}
Note that up to a trivial change of variables the double sine function coincides with the hyperbolic gamma function of Ruijsenaars \cite{Ruj1,Ruj3} (see particularly equation (3.52) of \cite{Ruj1}). The global properties of this function are covered by \cite[Prop. III.5]{Ruj3}, and since the exponential prefactor in  \eqref{when} does not affect these, this implies part (i).

The integral formula, part (v), is proved in \cite[Prop. 2]{N}. Property (ii) is then obvious by analytic continuation, but in any case this is a standard property of the double sine function:  see \cite[Appendix A]{JM}.

For part (iii) we only have to check one relation, by symmetry. The double sine function satisfies the difference relation
\[\frac{\sin_2(z+\omega_1\b\omega_1,\omega_2)}{\sin_2(z\b\omega_1,\omega_2)} = \frac{1}{2\sin (\pi z/\omega_2)}.\]
This can be found in \cite[Prop. III.1]{Ruj1} or \cite[Equation (A.8)]{JM}.
Combining this with the  identity
\[B_{2,2}(z+\omega_1\b\omega_1,\omega_2)-B_{2,2}(z\b\omega_1,\omega_2)=2B_{1,1}(z\b\omega_2)=\frac{2z}{\omega_2} - 1\]
gives  (iii). Alternatively one can give a direct proof using the integral identity (v).

The product expansion, part (iv),  is due to Shintani. It can be found in \cite[Corollary 6]{N} or \cite[Equation (3.58)]{Ruj1}.
 \end{proof}

\subsection{Triple sine with repeated argument}

We now consider another function of $z\in \bC$ and $\omega_1,\omega_2\in \bC^*$ related to the triple sine function. We define
\begin{equation}
\label{defg}G(z\b\omega_1,\omega_2)=e^{\frac{\pi i}{6} \cdot  B_{3,3}(z+\omega_1\b\omega_1,\omega_1,\omega_2)}
\cdot \sin_3\big(z+\omega_1\b \omega_1,\omega_1,\omega_2\big),\end{equation}
where $B_{3,3}(z\b\omega_1,\omega_2,\omega_3)$ is the multiple Bernoulli polynomial
\[B_{3,3}(z\b\omega_1,\omega_2,\omega_3)=\frac{z^3}{\omega_1\omega_2\omega_3} - \frac{3(\omega_1+\omega_2+\omega_3)}{2\omega_1\omega_2\omega_3} z^2 \]\[+\frac{\omega_1^2+\omega_2^2 +\omega_3^2+3\omega_1\omega_2+3\omega_2\omega_3 + 3\omega_3 \omega_1}{2\omega_1\omega_2\omega_3} z-\frac{(\omega_1+\omega_2+\omega_3)(\omega_1\omega_2+\omega_2\omega_3+\omega_3\omega_1)}{4\omega_1\omega_2\omega_3}.\]
Note that the function $G(z\b\omega_1,\omega_2)$  is not symmetric in $\omega_1,\omega_2$: we will use both the functions  $G(z\b\omega_1,\omega_2)$ and $G(z\b\omega_2,\omega_1)$ in what follows. 
  
 \begin{prop}
 \label{g}
 The function $G(z\b\omega_1,\omega_2)$ is a single-valued meromorphic function of variables   $z\in \bC$ and $\omega_1,\omega_2\in \bC^*$ under the assumption $\omega_1/\omega_2\notin \bR_{<0}$. It has the following properties:
 \begin{itemize}
 \item[(i)] The function is everywhere regular and vanishes only  at the points
 \[z=a\omega_1+b\omega_2, \quad a,b\in\bZ,\]
 with  $a<0$ and $b\leq 0$, or  $a>0$ and $b> 0$.\smallskip

\item[(ii)] It satisfies the symmetry relation
\begin{equation}
\label{nosven}\frac{\partial}{\partial \omega_2} \log G(z\b\omega_1,\omega_2)=\frac{\partial}{\partial \omega_1}\log G(z\b\omega_2,\omega_1),\end{equation}
and is invariant under simultaneous rescaling of all three arguments.\smallskip

\item[(iii)] It satisfies the difference relation
 \begin{equation}
 \label{diffyg}\frac{G(z+\omega_1\b\omega_1,\omega_2)}{G(z\b\omega_1,\omega_2)}=F(z+\omega_1\b\omega_1,\omega_2)^{-1}.\end{equation}

\item[(iv)] There is a relation
 \begin{equation}
\label{latehome}\frac{\partial}{\partial \omega_2}\log F(z\b\omega_1,\omega_2)=\frac{\partial}{\partial z} \log G(z\b\omega_2,\omega_1).\end{equation}

\item[(v)] When $\Re(\omega_i)>0$ and $-\Re(\omega_1)< \Re(z)<\Re(\omega_1+\omega_2)$ there is an integral representation \begin{equation}
\label{homealone}G(z\b\omega_1,\omega_2)=\exp\Bigg(\int_C \frac{-e^{(z+\omega_1)s}}{(e^{\omega_1 s}-1)^2(e^{\omega_2 s}-1)}\frac{ds}{s}\Bigg),\end{equation}
where as before, the contour $C$ follows the real axis from $-\infty$ to $+\infty$ avoiding the origin by a small detour in the upper half-plane.
\end{itemize}
\end{prop}

\begin{proof}
The global properties follow from standard properties of mutiple sine functions \cite{JM}, or can be deduced from the corresponding properties of the function $F(z\b\omega_1,\omega_2)$ using the relations \eqref{latehome}. The integral representation, part (v), is proved in \cite[Prop. 2]{N}, and  the relations  (ii) and (iv)  are then immediate  by differentiating under the integral sign and comparing with Prop. \ref{f}(v). Part (iii) follows directly from the integral formula.
 \end{proof}

\subsection{Reflection relations}

The following reflection properties will be needed later.

\begin{prop}
\label{ref}
When  $\Im(\omega_1/\omega_2)>0$ and $z\in \bC$ the following relations hold
\begin{equation}
\label{reff}F(z+\omega_2\b\omega_1,\omega_2)\cdot F(z\b\omega_1,-\omega_2)=  \prod_{k\geq 0} \big(1-x_2 q_2^{k}\big)\cdot \prod_{k\geq 1} \big(1-x_2^{-1} q_2^{k}\big)^{-1},\end{equation}
\begin{equation}
\label{refg}G(z+\omega_2\b\omega_1,\omega_2)\cdot G(z\b\omega_1,-\omega_2)= \prod_{k\geq 1} \big(1-x_2 q_2^{k}\big)^{k}\cdot \prod_{k\geq 1} \big(1-x_2^{-1} q_2^{k}\big)^{k},\end{equation}
where $x_2$ and $q_2$ are defined in \eqref{bore}.
\end{prop} 

\begin{proof}
We start with \eqref{reff}. Note first that the contour $C$ defining the integral representation \eqref{homeagain} can be rotated to point along a different ray $\ell=\bR_{>0} \cdot r$, providing that it does not hit the rays spanned by $2\pi i/\omega_i$ where the poles of the integrand lie, and providing that \begin{equation}
\label{l}\Re(\omega_i s)>0, \qquad 0<\Re(zs)<\Re((\omega_1+\omega_2)s),\end{equation}
 for $s\in \ell$, which ensures that the integrand decays exponentially as $|s|\to \infty$ with $s\in \pm \ell$.   For small enough $|z|$ the conditions \eqref{l}  are equivalent to the assumption that the  half-plane centered on $\ell$ contains the points $\omega_i^{-1}$ and $z^{-1}$.
 
For definiteness, we can assume that $z$ and  $\omega_1,\omega_2$ lie close but not on the positive real axis: since the right-hand side of our relation  defines an analytic function for $\Im(\omega_1/\omega_2)>0$, the result will follow in general by analytic continuation. The term $F(z+\omega_2\b\omega_12,\omega_2)$ is thus covered by the integral representation \eqref{homeagain}. To give a similar description of the term $F(z\b\omega_1,-\omega_2)$ we can use the homogeneity property of $F(z\b\omega_1,\omega_2)$ to rotate $z$, $\omega_1$ and $-\omega_2$ into the right-hand half-plane. The result is an integral representation whose contour is a rotation of $C$ which points along the negative imaginary axis. 
\begin{figure}
\begin{tikzpicture}[scale=0.8]


\draw (0,.5) arc (90:-90:.5);
\draw (130:.5) arc (130:-50:.5);

\draw[->] (-90:0.5) -- (-90:2.5);
\draw (-90:2.5) -- (-90:4.5);

\draw[->] (90:4.5) -- (90:2.5);
\draw (90:2.5) -- (90:0.5);

\draw[->] (130:4.5) -- (130:2.5);
\draw (130:2.5) -- (130:0.5);

\draw[->] (130:-0.5) -- (130:-2.5);
\draw (130:-2.5) -- (130:-4.5);

\draw (-2.1,1.8) node {$\scriptstyle C_+$};
\draw (.5,3) node {$\scriptstyle C_-$};


\draw[dashed] (90:4.5) arc (90:130:4.5);
\draw[dashed] (90:-4.5) arc (90:130:-4.5);


\draw[fill] (0,0) circle [radius=0.05];

\draw[fill] (110:1) circle [radius=0.05];
\draw[fill] (110:2) circle [radius=0.05];
\draw[fill] (110:3) circle [radius=0.05];
\draw[fill] (110:4) circle [radius=0.05];

\draw[fill] (110:-1) circle [radius=0.05];
\draw[fill] (110:-2) circle [radius=0.05];
\draw[fill] (110:-3) circle [radius=0.05];
\draw[fill] (110:-4) circle [radius=0.05];

\draw[fill] (70:1) circle [radius=0.05];
\draw[fill] (70:2) circle [radius=0.05];
\draw[fill] (70:3) circle [radius=0.05];
\draw[fill] (70:4) circle [radius=0.05];

\draw[fill] (70:-1) circle [radius=0.05];
\draw[fill] (70:-2) circle [radius=0.05];
\draw[fill] (70:-3) circle [radius=0.05];
\draw[fill] (70:-4) circle [radius=0.05];

\draw (20:1) circle [radius=0.05];
\draw (20:1) node {$\ \ \ \ \ \ \scriptstyle \omega_2^{-1}$};
\draw (-20:1) circle [radius=0.05];
\draw (-20:1) node {$\ \ \ \ \ \ \scriptstyle \omega_1^{-1}$};

\end{tikzpicture}
\caption{The contours for the proof of Proposition \ref{ref}.
\label{fig2}}
\end{figure}
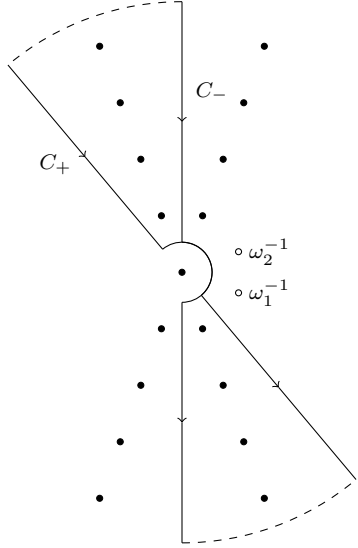

Consulting Figure \ref{fig2}   it is now easy to see that\[F(z+\omega_2\b\omega_1,\omega_2)\cdot F(z\b\omega_1,-\omega_2)=\exp\bigg( \int_{C_+} \frac{e^{(z+\omega_2)s}}{(e^{\omega_1 s}-1)(e^{\omega_2 s}-1)}\frac{ds}{s}+\int_{C_-} \frac{e^{zs}}{(e^{\omega_1 s}-1)(e^{-\omega_2 s}-1)}\frac{ds}{s}  \bigg),\]
where $C_-$ and $C_+$ are rotations of our standard contour $C$ whose positive directions lie along small clockwise, respectively anti-clockwise, perturbations of the ray $\bR_{<0} \cdot 2\pi i/\omega_2$. Since the integrands differ only by a sign, the expression in the exponential is just the sum of residues at the points $s=2\pi i m/\omega_2$ for $m\in \bZ\setminus \{0\}$, taken with a positive or negative sign depending on the sign of $m$. This residue is
\[2\pi i\cdot  \operatorname{Res}_{s=\frac{2\pi i m}{\omega_2}}\bigg( \frac{e^{(z+\omega_2)s}\,ds}{(e^{\omega_1 s}-1)(e^{\omega_2 s}-1)s}\bigg)= \frac{e^{2\pi i mz/\omega_2}}{m(e^{2\pi i m\omega_1/\omega_2}-1)}=\frac{x_2^{m} }{m(q^{m}_2-1)}.\]
Thus we obtain an expression
\[F(z+\omega_2\b\omega_1,\omega_2)\cdot F(z\b\omega_1,-\omega_2)=\exp\bigg(\sum_{m\geq 1}\frac{-x_2^m }{m(1-q^{m}_2)} +\sum_{m\geq 1}  \frac{x_2^{-m}q_2^{m}}{m(1-q_2^{m})} \bigg)\]\[=\exp\bigg(-\sum_{m\geq 1, k\geq 0}\frac{1}{m}{x_2^m q_2^{km}} +\sum_{m\geq 1,k\geq 1}  \frac{1}{m} {x_2^{-m} q_2^{km}} \bigg)=\prod_{k\geq 0} (1-x_2 q_2^{k})\cdot \prod_{k\geq 1} (1-x_2^{-1} q_2^{k})^{-1},\]
which completes the proof of \eqref{reff}.

To prove \eqref{refg} we follow the same strategy. Under the same conditions as before we get\[G(z+\omega_2\b\omega_1,\omega_2)\cdot G(z\b\omega_1,-\omega_2)\]\[=\exp\bigg( \int_{C_+} \frac{-e^{(z+\omega_1 +\omega_2)s}}{(e^{\omega_1 s}-1)^2(e^{\omega_2 s}-1)}\frac{ds}{s}+\int_{C_-} \frac{-e^{(z+\omega_1)s}}{(e^{\omega_1 s}-1)^2(e^{-\omega_2 s}-1)}\frac{ds}{s}  \bigg).\]
Once again the integrands differ only by a sign, so the expression in the exponential is just the sum of the residues at the points $s=2\pi i m/\omega_2$ for $m\in \bZ\setminus \{0\}$, taken with a positive or negative sign depending on the sign of $m$. This time the residue is
\[2\pi i\cdot  \operatorname{Res}_{s=\frac{2\pi i m}{\omega_2}}\bigg( \frac{-e^{(z+\omega_1+\omega_2)s}\,ds}{(e^{\omega_1 s}-1)^2(e^{\omega_2 s}-1)s}\bigg)= \frac{-e^{2\pi i m(z+\omega_1)/\omega_2}}{m(e^{2\pi i m\omega_1/\omega_2}-1)^2}=\frac{-x_2^{m} q_2^{m}}{m(1-q_2^{m})^2}.\]
Thus we obtain an expression
\[G(z+\omega_2\b\omega_1,\omega_2)\cdot G(z\b\omega_1,-\omega_2)=\exp\bigg(\sum_{m\geq 1}\frac{-x_2^m q_2^{m}}{m(1-q_2^{m})^2} +\sum_{m\geq 1}  \frac{-x_2^{-m}q_2^{m}}{m(1-q_2^{m})^2} \bigg)\]\[=\exp\bigg(\sum_{m\geq 1, k\geq 1}-\frac{k}{m}{x_2^m q_2^{km}} -\sum_{m\geq 1,k\geq 1}  \frac{k}{m} {x_2^{-m} q_2^{km}} \bigg)=\prod_{k\geq 1} (1-x_2 q_2^{k})^{k}\cdot  \prod_{k\geq 1} (1-x_2^{-1} q_2^{k})^{k},\]
which completes the proof.
\end{proof}

\subsection{Polylogarithm and zeta identities}

The asymptotic expansions of the functions $F$ and $G$ which we  derive in Sections \ref{sss} and \ref{sst} below involve the polylogarithm and
 Riemann zeta functions. In this section we collect some simple integral identities involving  these functions.

 For all $k\in \bZ$ the polylogarithm $\Li_k(x)$ is defined by the power series
\begin{equation}\label{po}\Li_k(x)=\sum_{n\geq 1} \frac{x^n}{n^k},\end{equation}
which is absolutely convergent in the unit disc.
For $k\leq 0$ the function $\Li_k(x)$ is rational, and regular except for a pole at $x=1$. For  $k\geq 1$ the function $\Li_k(x)$ has a single logarithmic singularity at $x=1$.  We list the special cases\[ \Li_{-1}(x)= \frac{x}{(1-x)^2}, \quad \Li_0(x)=\frac{x}{1-x}, \quad \Li_{1}(x)=-\log(1-x).\]
In what follows we shall only use expressions of the form $\Li_k(e^{2\pi i a})$, and will always assume that  $\Im(a)>0$, so the power series \eqref{po} will suffice to define the polylogarithm, and the multi-valuedness of the analytic continuation of $\Li_k(x)$ for $k\geq 1$ will play no role.
 
\begin{prop}
\label{sme}
Take complex numbers $z$ and $\omega_1$ satisfying $0<\Re(z)<\Re(\omega_1)$ and $\Im(z/\omega_1)>0$. Then for each integer $d\in \bZ$ there is an expression
\begin{equation}
\label{hallow} \int_{C} \frac{e^{zs}\cdot  s^{-d}}{e^{\omega_1 s}-1} \,ds=\Big(\frac{\omega_1}{2\pi i}\Big)^{d-1}\cdot \Li_{d}(e^{2\pi i z/\omega_1}) ,\end{equation}
where as before the contour $C$ follows the real axis from $-\infty$ to $+\infty$, with a small detour around the origin in the upper half-plane. \end{prop}

\begin{proof}
The integrand has poles at the points $2\pi i n/\omega_1$ for $n\in \bZ$, with residues
\[ 2\pi i\cdot \operatorname{Res}_{s=\frac{2\pi in}{\omega_1}} \Bigg(\frac{e^{zs}}{e^{\omega_1 s}-1} \cdot\frac{ds}{s^d}\Bigg)=\Big(\frac{\omega_1}{2\pi i}\Big)^{d-1}\cdot \frac{e^{2\pi i n z/\omega_1}}{n^d}.\]
Note that the assumption $\Im(z/\omega_1)>0$ ensures
that the power series expansion \eqref{po} defining $\Li_d(e^{2\pi i z/\omega_1})$ is absolutely convergent, and then the right-hand side of \eqref{hallow} is $(2\pi i)$ times the sum of the residues of the poles in the upper half-plane.
To give a rigorous proof of \eqref{hallow} we first note that since the integrand decays exponentially as $|\Re(s)|\to \infty$ there is a relation
\begin{equation}
\label{sey}\int_{C_0} \frac{e^{zs}\cdot  s^{-d}}{e^{\omega_1 s}-1} ds- \int_{C_N} \frac{e^{zs}\cdot  s^{-d}}{e^{\omega_1 s}-1} ds=\Big(\frac{\omega_1}{2\pi i}\Big)^{d-1}\cdot \sum_{n=1}^N \frac{e^{2\pi i n z/\omega_1}}{n^d},\end{equation}
where for each $N\geq 0$ we denote by $C_N$ the shifted contour $C+2N\pi i/\omega_1$. 
But 
\[\int_{C_N} \frac{e^{zs}\cdot  s^{-d}}{e^{\omega_1 s}-1} ds=e^{2\pi i Nz/\omega_1}\cdot N^{-d}\cdot \int_{C_0} \frac{e^{zs}}{e^{\omega_1 s}-1}\cdot  \bigg(\frac{s}{N}+\frac{2\pi i}{\omega_1}\bigg) ^{-d}ds.\]
The integral on the right can easily be bounded independently of $N$, so using the hypothesis $\Im(z/\omega_1)>0$ again, we conclude that
the integral over $C_N$ in \eqref{sey} tends to 0 as $N\to \infty$. 
 \end{proof}

Note that differentiating \eqref{hallow} gives the relation
\begin{equation}
\label{glob}-\int_{C} \frac{e^{(z+\omega_1)s}\cdot  s^{1-d}}{(e^{\omega_1 s}-1)^2} ds=\frac{d}{d\omega_1} \bigg(\Big(\frac{\omega_1}{2\pi i}\Big)^{d-1}\cdot \Li_{d}(e^{2\pi i z/\omega_1})\bigg).\end{equation}

We shall also need an analogue of \eqref{glob}  for $z=0$ which involves the Riemann zeta function. Recall that  $\zeta(x)$ is a meromorphic function of $x\in \bC$ which is regular except for a simple pole at $x=1$, and satisfies
\[\zeta(k)=\sum_{n\geq 1} \frac{1}{n^k}=\Li_k(1),\]
for integers $k\geq 2$. For integers $k\geq 0$  one has  \begin{equation}
\label{negzeta}\zeta(-k)=\frac{(-1)^k \cdot B_{k+1}}{k+1},\end{equation}
where $B_{k+1}$ denotes the $(k+1)$st Bernoulli number. 

\begin{prop}
\label{useful}
Take $\omega_1\in \bC^*$ with $\Re(\omega_1)>0$. Then for all $d\in \bZ$  there is a relation
\begin{equation}
\label{spring} -\int_{C} \frac{e^{\omega_1 s}\cdot  s^{1-d}}{(e^{\omega_1 s}-1)^2} \, ds =\frac{(d-1)\cdot \zeta(d)}{2\pi i}\cdot \Big(\frac{\omega_1}{2\pi i}\Big)^{d-2}, \end{equation}
where the contour $C$ is as in Propositon \ref{sme}, and in the case $d=1$  the right-hand side of \eqref{spring} is defined by setting $(d-1)\cdot \zeta(d)=1$. \end{prop}

\begin{proof}
When $d\geq 2$ the argument of Proposition \ref{sme} also applies with $\Re(z)=0$ and hence yields
\[\int_{C} \frac{s^{-d}}{e^{\omega_1 s}-1} ds=\Big(\frac{\omega_1}{2\pi i}\Big)^{d-1}\cdot \zeta(d). \]
The result then follows by differentiating with respect to $\omega_1$.

When $d<0$ the integrand is regular at $s=0$ so we may replace the integral along $C$ by one along $\bR$. The symmetry under $s\leftrightarrow -s$ then forces the integral to be zero unless $d$ is odd. In that case the standard integral representation of the zeta function together with the duplication formula shows that 
\[2\int_{0}^\infty \frac{s^{-d}}{e^{\omega_1 s}-1} ds=\Big(\frac{\omega_1}{2\pi i}\Big)^{d-1}\cdot \zeta(d).\]
Differentiating with repsect to $\omega_1$ then gives
\[-\int_{C} \frac{e^{\omega_1 s}\cdot  s^{1-d}}{(e^{\omega_1 s}-1)^2} ds =-2\int_0^\infty \frac{e^{\omega_1 s}\cdot  s^{1-d}}{(e^{\omega_1 s}-1)^2}  ds=\frac{(d-1)}{2\pi i}\cdot \Big(\frac{\omega_1}{2\pi i}\Big)^{d-2}\cdot \zeta(d).\]

When $d=0$ the identity \eqref{spring} can be checked by a simple residue calculation. Since the integrand is invariant under $s\leftrightarrow -s$ we can combine the integral over $C$ and $-C$ to obtain
\[-\int_{C} \frac{e^{\omega_1 s}\cdot  s}{(e^{\omega_1 s}-1)^2} ds =\frac{1}{2}\cdot (2\pi i)\operatorname{Res}_{s=0} \bigg(\frac{e^{\omega_1 s} \cdot s}{(e^{\omega_1 s}-1)^2}  ds\bigg)=\frac{1}{2}\cdot \frac{2\pi i}{\omega_1^2}.\]
Since $\zeta(0)=-1/2$ this agrees with \eqref{spring}.
Finally, when $d=1$ the integrand has an obvious primitive and we obtain
\[-\int_{C} \frac{e^{\omega_1 s}}{(e^{\omega_1 s}-1)^2} ds=\frac{1}{\omega_1}\cdot \bigg[\frac{1}{e^{\omega_1 s}-1}\bigg]_{-\infty}^{\infty}=\frac{1}{\omega_1},\]
which matches with our definition of the right-hand side of \eqref{spring} in this case.
\end{proof}

\subsection{Asymptotic expansions as $\omega_2\to 0$}
\label{sss}

In this section we  give asymptotic expansions for the  functions $F$ and $G$ as the parameter $\omega_2\to 0$. 

\begin{prop}
\label{ass}
Fix $z\in \bC$ and $\omega_1\in \bC^*$ with  $0<\Re(z)<\Re(\omega_1)$ and $\Im(z/\omega_1)>0$. Then  there are asymptotic expansions
 \begin{equation}\label{fass}\log F(z\b\omega_1,\omega_2)\sim 
\sum_{k\geq 0} \frac{B_k\cdot \omega_2^{k-1} }{k!}  \cdot \Big(\frac{2\pi i }{\omega_1}\Big)^{k-1}\cdot  \Li_{2-k}(e^{2\pi i z/\omega_1}),\end{equation}
\begin{equation}\label{gas}\log G(z\b\omega_1,\omega_2)\sim  \sum_{k\geq 0} \frac{B_k\cdot \omega_2^{k-1} }{k!}  \cdot \frac{d}{d\omega_1} \bigg(\Big(\frac{2\pi i }{\omega_1}\Big)^{k-2}\cdot  \Li_{3-k}(e^{2\pi i z/\omega_1})\bigg),\end{equation}
\begin{equation}\label{gass}\log G(z\b\omega_2,\omega_1)\sim \sum_{k\geq 0} \frac{(k-1)\cdot B_k\cdot \omega_2^{k-2} }{k!}  \cdot\Big(\frac{2\pi i }{\omega_1}\Big)^{k-2}\cdot  \Li_{3-k}(e^{2\pi i z/\omega_1}),\end{equation}
valid as $\omega_2\to 0$ in any closed subsector $\Sigma$ of the half-plane $\Re(\omega_2)>0$.
\end{prop}

\begin{proof}
We focus first on \eqref{fass}, the other parts will then follow by a similar argument.  Using  the integral formula \eqref{homeagain} and the Laurent expansion\begin{equation}
\label{tay}\frac{1}{e^{\omega_2 s}-1}=\sum_{k\geq 0} \frac{B_k  \cdot (\omega_2 s)^{k-1}}{k!}= \frac{1}{\omega_2 s} - \frac{1}{2} + \frac{\omega_2 s}{12} + \cdots,\end{equation}
gives an expression
\[\log F(z\b \omega_1,\omega_2)=\int_C \frac{e^{zs}}{(e^{\omega_1 s}-1)(e^{\omega_2 s}-1)}\frac{ds}{s} =\int_C  \, \sum_{k\geq 0} \frac{ B_k \cdot \omega_2^{k-1} \cdot s^{k-1} }{k!}  \cdot \frac{e^{zs}  }{e^{\omega_1 s}-1}\cdot \frac{ds}{s}.\]
The result then follows formally by exchanging the order of integration and summation and using the identity \eqref{hallow}.

To justify this we must prove that for each integer $N> 0$
\begin{equation}
\label{frac}\frac{1}{\omega_2^{N-1}}\int_C \bigg(\frac{1}{e^{\omega_2 s}-1} - \sum_{k=0}^N \frac{ B_k \cdot (\omega_2 s)^{k-1} }{k!}\bigg) \cdot \frac{e^{zs}  }{e^{\omega_1 s}-1}\cdot \frac{ds}{s} \to 0\end{equation}
as $\omega_2\to 0$ in the closed subsector $\Sigma$.
Since $F$ is invariant under rescaling all variables we can assume that $|\omega_1|<1$. Let us rewrite the left-hand side of \eqref{frac} as \begin{equation}\label{flo}\omega_2\cdot I(\omega_2)=\omega_2\cdot \int_C R_N(\omega_2 s) \cdot \frac{ s^{N} \cdot e^{zs}  }{e^{\omega_1 s}-1}\cdot \frac{ds}{s},\end{equation}
where $R_N$ denotes the meromorphic function
\[R_N(x)=\frac{1}{x^{N}}\cdot \bigg(\frac{1}{e^x -1} - \sum_{k=0}^N \frac{ B_k \cdot x^{k-1} }{k!}\bigg).\]
We are reduced to proving that the integral $I(\omega_2)$ is bounded as $\omega_2\to 0$ in $\Sigma$.

The  function $R_N(x)$ is regular on the unit disc and on $\Sigma$, and tends to 0 as $|x|\to \infty$ with $\pm x\in \Sigma$. Thus there is a constant $K>0$ such that
\[\pm x\in \Sigma\text{ or } |x|<1\implies |R_N(x)|<K.\]
Since we assumed that $|\omega_1|<1$, we can take the contour $C$ in \eqref{flo} to consist of the union of the segments $(-\infty,1)$ and $(1,\infty)$ of the real axis, together with the intersection of the unit circle with the upper half-plane.
 It follows that
\[s\in C\text{ and } \omega_2\in \Sigma \text{ with }|\omega_2|<1 \implies |R_N(\omega_2 s)|<K.\]
This then gives a bound
\[|I(\omega_2)|<K\cdot \int_C \,\bigg|\frac{ s^{N} \cdot e^{zs}  }{e^{\omega_1 s}-1}\cdot \frac{ds}{s}\bigg|<\infty,\]
independently of $\omega_2$.
This completes the proof of \eqref{fass}.

The other two expansions can be derived in exactly the same way. For \eqref{gas} we use the same Laurent expansion \eqref{tay} and the  identity \eqref{glob},
and for \eqref{gass} we use the Laurent series
\begin{equation}
\label{sinh}\frac{e^{\omega_2 s}}{(e^{\omega_2 s}-1)^2}=
\sum_{k\geq 0} \frac{(1-k)\cdot B_k\cdot (\omega_2 s)^{k-2}}{k!}\\=\frac{1}{(\omega_2 s)^2}-\frac{1}{12}+\frac{1}{240}(\omega_2 s)^2+\cdots  \end{equation}
obtained by differentiating \eqref{tay}, together with the identity \eqref{hallow}.
\end{proof}

Note that the  expansions of Proposition \ref{ass} are related   by the identities \eqref{nosven} and  \eqref{latehome}. We shall also need the following analogues of the expansions \eqref{gas} and \eqref{gass} when $z=0$.

\begin{prop}
\label{assy}
Fix $\omega_1\in \bC^*$ with $\Re(\omega_1)>0$. Then there are asymptotic expansions
\begin{equation}\label{gassi}\log G(0\b\omega_1,\omega_2)\sim \frac{\zeta(3)}{\pi i}\cdot \frac{\omega_1}{2\pi i \omega_2}  +\frac{\pi i}{24}  +\sum_{k\geq 2} \frac{ B_k\cdot B_{k-2}}{(2\pi i)\cdot k!}  \cdot  \Big(\frac{2\pi i \omega_2}{\omega_1}\Big)^{k-1},\end{equation}
\begin{equation}\label{gassi2}\log G(0\b\omega_2,\omega_1)\sim -\zeta(3)  \cdot\Big(\frac{\omega_1}{2\pi i\omega_2}\Big)^2+ \frac{1}{12}\log\Big(\frac{\omega_1}{\omega_2}\Big)\\- \sum_{k\geq 3} \frac{B_k\cdot   B_{k-2}}{k\cdot (k-2)!\cdot (k-2)}  \cdot  \Big(\frac{2\pi i \omega_2}{\omega_1}\Big)^{k-2},\end{equation}
valid as $\omega_2\to 0$ in any closed subsector of the half-plane $\Re(\omega_2)>0$.
\end{prop}

\begin{proof}
The expansion \eqref{gassi}  is proved in exactly the same way as \eqref{gas}, replacing the identity \eqref{glob} with \eqref{spring}, and using \eqref{negzeta} in the form
\begin{equation}
\label{negzeta2}
(2-k)\cdot \zeta(3-k)=(-1)^k \cdot B_{k-2}, \qquad k\geq 3,
\end{equation}
together with the well-known identity $\zeta(2)=\pi^2/6$. To prove  \eqref{gassi2} we first apply the argument of Proposition \ref{ass} to the integral 
\[\frac{\partial}{\partial \omega_1} \log G(0\b\omega_2,\omega_1)=\int_C \frac{e^{(\omega_1+\omega_2) s}\, ds}{(e^{\omega_1 s}-1)^2(e^{\omega_2 s}-1)^2},\]
using the Laurent series \eqref{sinh} and the identity \eqref{spring}. This gives
\begin{equation}
\label{done} \frac{\partial}{\partial \omega_1} \log G(0\b\omega_2,\omega_1)\sim \frac{\zeta(3)}{2\pi ^2}\cdot \frac{\omega_1}{\omega_2^2}+\frac{1}{12\omega_1}\\   +\sum_{k\geq 3} \frac{B_k\cdot  (k-1)(2-k)\cdot \zeta(3-k)}{(2\pi i) \cdot k!\cdot \omega_2}  \cdot  \Big(\frac{2\pi i\omega_2}{\omega_1}\Big)^{k-1}.\end{equation}
Integrating term-by-term  and using the identity \eqref{negzeta} then gives the result. The constant of integration is determined by the condition that $G(z\b\omega_1,\omega_2)$ is invariant under rescaling of all arguments.
 \end{proof}

\subsection{Asymptotic expansions as $\omega_2\to \infty$}
\label{sst}
We shall also need the asymptotic expansions of the functions $F$ and $G$  as $\omega_2\to \infty$. These involve the Bernoulli polynomials $B_n(x)$, which can be defined by the Laurent expansion \eqref{tay2} below.

\begin{prop}
\label{inflimit}
Fix $z\in \bC$ and $\omega_1\in \bC^*$ satisfying   $\Im(z/\omega_1)>0$. Then as $\omega_2\to \infty$ in any closed subsector $\Sigma$ of the half-plane $\Re(\omega_2)>0$ there are asymptotic expansions
 \begin{equation}\label{fass2}\log F(z\b\omega_1,\omega_2)\sim -\frac{\pi i}{12}\cdot \frac{ \omega_2}{\omega_1}+B_1(z/\omega_1)\cdot \log(\omega_2)+O(1),
\end{equation}
\begin{equation}\label{gas2}\log G(z\b\omega_1,\omega_
2)\sim  \frac{\zeta(3)}{4\pi^2}\cdot \frac{\omega_2^2}{\omega_1^2} + \frac{\pi i}{12}\cdot \frac{ z\omega_2}{\omega_1^2} + \frac{1}{2} \log(\omega_2)\cdot \frac{d}{d\omega_1} (\omega_1 \cdot B_2(z/\omega_1))+O(1),
\end{equation}
\begin{equation}\label{gass2}\log G(z\b\omega_2,\omega_1)\sim -\frac{\zeta(3)}{2\pi^2} \cdot \frac{\omega_2}{\omega_1}  -{\frac{\pi i}{12}\cdot   B_1(z/\omega_1)} + \sum_{k\geq 2}  \frac{(-1)^{k}\cdot B_k (z/\omega_1)\cdot B_{k-2}}{k!\cdot (2\pi i)} \cdot \bigg(\frac{2\pi i\omega_1}{\omega_2}\bigg)^{k-1}.\end{equation}
For \eqref{fass2} we need to assume $\Re(z)>0$,  and for \eqref{gas2} that $\Re(z+\omega_1)>0$.
\end{prop}

\begin{proof}
Let us start with \eqref{gass2}. Using the integral representation \eqref{homealone} we have
\[\log  G(z\b\omega_2,\omega_1)=\int_C \frac{-e^{(z+\omega_2) s}}{(e^{\omega_1 s}-1)(e^{\omega_2 s}-1)^2}\cdot\frac{ds}{s}.\]
Note that this is valid whenever $-\Re(\omega_2)<\Re(z)<\Re(\omega_1+\omega_2)$, a condition  which holds automatically for sufficiently large  $|\omega_2|$ under the assumption $\omega_2\in \Sigma$. 
Applying the 
Laurent expansion
\begin{equation}
\label{tay2}\frac{e^{z s}}{e^{\omega_1 s}-1}=\sum_{k\geq 0} \frac{B_k (z/\omega_1) \cdot (\omega_1 s)^{k-1}}{k!},\end{equation}
gives an expression
\[ \log  G(z\b\omega_2,\omega_1)=\int_C\, \sum_{k\geq 0}  \frac{B_k (z/\omega_1) \cdot (\omega_1 s)^{k-1}}{k!}\cdot \frac{-e^{\omega_2 s}} {(e^{\omega_2 s}-1)^2} \cdot \frac{ds}{s}.\]
Exchanging the order of integration and summation, and using \eqref{spring} gives
\[\log  G(z\b\omega_2,\omega_1)\sim  \sum_{k\geq 0}  \frac{B_k (z/\omega_1)}{k!}\cdot   \Big(\frac{2\pi i \omega_1}{\omega_2}\Big)^{k-1} \cdot \frac{(2-k)\cdot \zeta(3-k)}{2\pi i},\]
where as usual we set $(2-k)\cdot \zeta(3-k)=1$  when $k=2$.  Using the identity \eqref{negzeta2} this expression reduces to \eqref{gass2}.

To justify this we must show that for $N\gg 0$
\begin{equation}\label{7a+}\omega_2^{N-1}\cdot \int_C \Bigg(\frac{e^{zs}}{e^{\omega_1 s}-1} - \sum_{k\geq 0}^N  \frac{B_k (z/\omega_1)\cdot (s\omega_1)^k}{k!}\Bigg)\cdot   \frac{-e^{\omega_2 s}}{(e^{\omega_2 s}-1)^2} \cdot \frac{ds}{s}\to 0\end{equation}
as $\omega_2\to \infty$ in the closed subsector $\Sigma$ of the right-hand half-plane. We can rewrite the left-hand side of this expression as
\[\frac{1}{\omega_2}\cdot I(\omega_2)=\frac{1}{\omega_2}\cdot \int_C \frac{R_N(s)}{s^N}\cdot \frac{-e^{\omega_2 s}\cdot (\omega_2 s)^{N}}{(e^{\omega_2 s}-1)^2} \cdot \frac{ds}{s},\]
where $R_N(s)$ denotes the expression in brackets in \eqref{7a+}. We must show that $I(\omega_2)$ is bounded for $\omega_2\in \Sigma$ with $|\omega_2|\gg 0$.

When $|\omega_2|>1$ we can  replace the contour $C$ by $|\omega_2|^{-1}\cdot C$ without changing the value of $I(\omega_2)$. The function $f(s)=R_N(s)/s^N$ is regular near $s=0$ and is bounded on the real axis as $|s|\to \infty$. We can therefore  find a bound $|f(s/|\omega_2|)|<K$ valid for $s\in C$ and $|\omega_2|>1$.  Replacing $s$ by $s/|\omega_2|$ then gives
\begin{equation}
\label{mot2}|I(\omega_2)|<K\cdot \int_C \bigg| \frac{e^{\eta s}\cdot (\eta s)^{N}}{(e^{\eta s}-1)^2} \cdot \frac{ds}{s}\bigg|,\end{equation}
where $\eta=\omega_2/|\omega_2|$. Since $\eta$  lies on the compact subset of $\bC^*$ consisting of the intersection of the unit circle with the sector $\Sigma$, we can find a uniform bound for the integral appearing on the right of \eqref{mot2}, which gives the claim.

For \eqref{fass2} we apply the same argument to the integral
\[\frac{\partial}{\partial \omega_2} \log  F(z\b\omega_1,\omega_2)=\int_C \frac{-e^{(z+\omega_2) s}\, ds}{(e^{\omega_1 s}-1)(e^{\omega_2 s}-1)^2}\]
to obtain an expansion
\[\frac{\partial}{\partial \omega_2} \log  F(z\b\omega_1,\omega_2)\sim  
-\frac{\pi i}{12\omega_1}+\frac{B_1(z/\omega_1)}{\omega_2} + \sum_{k\geq 2}  \frac{(-1)^{k-1}\cdot B_k (z/\omega_1)\cdot B_{k-1}}{k!\cdot \omega_2} \cdot \bigg(\frac{2\pi i\omega_1}{\omega_2}\bigg)^{k-1}.\]
Integrating with respect to $\omega_2$ gives \eqref{fass2}, with the constant of integration again determined by homogeneity. Similarly, for \eqref{gas2} we use the integral
\[\frac{\partial}{\partial \omega_2} \log  G(z\b\omega_1,\omega_2)=\int_C \frac{e^{(z+\omega_1+\omega_2) s}\, ds}{(e^{\omega_1 s}-1)^2(e^{\omega_2 s}-1)^2}\]
together with \eqref{spring} and the 
 Laurent expansion
\begin{equation}
\frac{-e^{(z+\omega_1)s}}{(e^{\omega_1 s}-1)^2}=\sum_{k\geq 0} \frac{s^{k-2}}{k!} \cdot \frac{d}{d\omega_1} \big(B_k (z/\omega_1)   \omega_1^{k-1}\big),\end{equation}
obtained by differentiating \eqref{tay2}, to get an expression
\[\frac{\partial}{\partial \omega_2} \log  G(z\b\omega_1,\omega_2)=\frac{\zeta(3)}{2\pi^2} \cdot \frac{\omega_2}{\omega_1^2} +\frac{\pi i}{12}\cdot\frac{z}{\omega_1^2}
+\sum_{k\geq 2} \frac{(-1)^{k} \cdot B_{k-2}\cdot (2\pi i)^{k-2} }{ k!\cdot \omega_2^{k-1}}\cdot \frac{d}{d\omega_1}\big(B_k(z/\omega_1)\cdot \omega_1^{k-1}\big).\]
Integrating with respect to $\omega_2$ then gives \eqref{gas2}.
\end{proof}

\section{Solution to the Riemann-Hilbert problem}
\label{solution}

In this section we solve the conifold Riemann-Hilbert problems of Section \ref{bps} using the special functions $F$ and $G$ introduced in the last section.

\subsection{Exponential pre-factors}

Let us again take $\omega_1,\omega_2\in \bC^*$ with $\omega_1/\omega_2\notin \bR_{<0}$. We consider the following meromorphic function of $z\in \bC$ 
\begin{equation}
\label{defh}H(z\b \omega_1,\omega_2)=\frac{{G}(z\b \omega_1,\omega_2)}{ G(0\b  \omega_1,\omega_2)}.\end{equation}
This function, together with $F(z\b\omega_1,\omega_2)$  will form the basis of the solution to the Riemann-Hilbert problem which we give in the next subsection.
However, to give the correct limiting behavior as $\omega_2\to 0$ and  $\omega_2\to \infty$ we first need to modify them by some exponential prefactors.

Under the assumption $\Im(z/\omega_1)>0$ we define
\begin{gather}
\label{mod}F^*(z\b\omega_1,\omega_2)=F(z\b\omega_1,\omega_2) \cdot e^{Q_F(z\b\omega_1,\omega_2)},\\ H^*(z\b\omega_1,\omega_2)=H(z\b\omega_1,\omega_2) \cdot e^{Q_H(z\b\omega_1,\omega_2)},\end{gather}
where $Q_F$ and $Q_H$ are Laurent polynomials in $\omega_2$ given explicitly by
\begin{equation*}
Q_F(z\b \omega_1,\omega_2)=  -\frac{\omega_1}{2\pi i \omega_2}\cdot \Li_2(e^{2\pi i z/\omega_1}) -\frac{1}{2}\log(1-e^{2\pi i z/\omega_1})+\frac{\pi i }{12}\cdot\frac{\omega_2}{\omega_1},\end{equation*}
\begin{equation*}Q_H(z\b \omega_1,\omega_2)= \frac{d}{d\omega_1} \bigg({\frac{1}{\omega_2}\Big(\frac{\omega_1}{2\pi i}\Big)^2 \cdot {(\zeta(3)}-\Li_3(e^{2\pi i z/\omega_1}))}+\frac{\omega_1}{4\pi i } (\Li_2(e^{2\pi i z/\omega_1})-\zeta(2))\bigg) -\frac{\pi i}{12}\cdot\frac{z\omega_2}{\omega_1^2}.\end{equation*}
These expressions are uniquely determined by the asymptotic properties of the resulting functions $F^*$ and $H^*$ (see the proof of Theorem \ref{xmas} below). 


\begin{prop}
\label{need}
Take $z\in \bC$ and $\omega_1,\omega_2\in \bC^*$ satisfying $\omega_1/\omega_2\notin \bR_{<0}$ and $\Im(z/\omega_1)>0$. 
The functions $F^*$ and $H^*$ introduced above satisfy the difference relations
\begin{equation}\label{early}\frac{F^*(z+\omega_1\b\omega_1,\omega_2)}{F^*(z\b\omega_1,\omega_2)}=\frac{1}{1-x_2},\end{equation}
 \begin{equation}\label{late}\frac{H^*(z+\omega_1\b\omega_1,\omega_2)}{H^*(z\b\omega_1,\omega_2)}=F^*(z+\omega_1\b\omega_1,\omega_2)^{-1}.\end{equation}
When $\Im(\omega_1/\omega_2)>0$ there are  also reflection relations
\begin{equation}
\label{reff2}F^*(z\b\omega_1,\omega_2)\cdot F^*(z\b\omega_1,-\omega_2)=  \prod_{k\geq 0} \big(1-x_2 q_2^{k}\big)\cdot \prod_{k\geq 1} \big(1-x_2^{-1} q_2^{k}\big)^{-1},\end{equation}
\begin{equation}
\label{refg2}H^*(z\b\omega_1,\omega_2)\cdot H^*(z\b\omega_1,-\omega_2)= {\prod_{k\geq 1} \big(1-x_2 q_2^{k}\big)^{k}}\cdot{ \prod_{k\geq 1} \big(1-x_2^{-1} q_2^{k}\big)^{k}}\cdot \prod_{k\geq 1}(1-q_2^k)^{-2k},\end{equation}
where $x_2$ and $q_2$ are defined in \eqref{bore}.
\end{prop}

\begin{proof}
Relations \eqref{early} and \eqref{late}   follow directly from the corresponding relations \eqref{diffy} and \eqref{diffyg} for $F$ and $G$.  One just needs to check that
\[Q_F(z+\omega_1\b\omega_1,\omega_2)=Q_F(z\b\omega_1,\omega_2),\]
\[Q_H(z+\omega_1\b\omega_1,\omega_2)-Q_H(z\b\omega_1,\omega_2)=-Q_F(z\b\omega_1,\omega_2),\]
but this is easily done. Note that the denominator in \eqref{defh} has no effect because it is constant in $z$. 

Similarly, the  relations \eqref{reff2} and \eqref{refg2}  follow from the  corresponding reflection properties in Proposition \ref{ref}. Note that in \eqref{reff} one has $F(z+\omega_2\b\omega_1,\omega_2)$ rather than simply $F^*(z\b\omega_1,\omega_2)$ as in \eqref{reff2}, and similarly for $H$. However this effect is precisely cancelled by the constant term of the Laurent polynomials $Q_F$ and $Q_H$. In detail the relations are
\[\log F(z+\omega_2\b  \omega_1,\omega_2) - \log F(z\b  \omega_1,\omega_2) = -\log (1-e^{2\pi i z/\omega_1}), \]\[ \log G(z+\omega_2\b  \omega_1,\omega_2)-\log G(z\b  \omega_1,\omega_2)=\frac{d}{d\omega_1} \bigg(\frac{\omega_1}{2\pi i}\cdot\Li_2(e^{2\pi i z/\omega_1})\bigg).\]
The first is immediate from the second  relation of  \eqref{diffy}, whereas the second follows from the integral representation \eqref{homealone} and  the identity  \eqref{glob}.
\end{proof}

\subsection{The solution}
 
We can now give the solution to our Riemann-Hilbert problem.

\begin{thm}
\label{xmas}
The unique solution to Problem \ref{tre} is
\begin{equation}
\label{ji}B(v,w,t)={F^*}(v\b  w,-t), \qquad  D(v,w,t)=H^*(v\b w,-t),\end{equation}
where the functions $F^*$ and $H^*$ are defined in the previous subsection.
\end{thm}

\begin{proof}
We must first check that for a fixed $(z,w)\in M_+$ these formulae do indeed define non-vanishing holomorphic functions on the required domain
$\cV(0)=\cH(-1)\cup \cH(0)$. Recall that $\cH(-1)$ or $\cH(0)$ are the half-planes centered on the rays $\bR_{>0}\cdot 2\pi i (v-w)$ and $\bR_{>0}\cdot 2\pi i v$ respectively. Note that $w\notin\cV(0)$ because
\[\Re(2\pi i (v-w)/w)=\Re(2\pi i v/w)<0.\]

Suppose that $t\in \cV(0)$ is a zero or pole of ${F}(v\b w,-t)$.
Proposition \ref{f}(i) implies that $\mp v= aw-bt$ with $a,b\in \bZ_{\geq 0}$. Since $\Im(v/w)> 0$, such a relation implies that $b>0$. Moreover, consulting Proposition \ref{f}(i) more carefully, we see that  $a>0$. We conclude that $t$ lies on the ray  spanned by $\pm v +aw$ for some $a\geq 1$. 

Now if $t$  lies in the half-plane $\cH(0)$  then rotating $\cH(0)$ by  small angle gives a half-plane containing $\pm v+aw$ and $\mp v$, but not $w$, a contradiction. On the other hand, if $t$ is contained in $\cH(-1)$ then rotating by a small angle gives a half-plane containing $\pm v+aw$ and $\mp (v-w)$, but not $w$, another contradiction. 

We conclude that ${F}(v\b w,-t)$ is holomorphic and non-vanishing for $t\in \cV(0)$. The same argument applies to $G( v\b w,-t)$ using Proposition \ref{g}(i). Finally, note that the denominator in \eqref{defh} causes no trouble, since by Proposition \ref{g}(i) again,   $G(0\b  w,-t)$ is regular and non-vanishing whenever $t\notin \bR_{>0}\cdot w$.

We now check the conditions of Problem \ref{tre} one by one. Parts (iii) and (iv)  follow immediately from Proposition \ref{need}, so it remains to prove the asymptotic properties (i) and (ii). Let  $\Sigma$ be a closed subsector of  $\cV(0)$. We must show that  for $t\in \Sigma$ one has
\begin{itemize}
\item[(i)]
$F^*(v\b w,-t)\text{ and } H^*(v\b w,-t)\to 1 \text { as } \omega_2\to 0$,

\item[(ii)] there exists $k>0$ such that for all $|\omega_2|\gg 0$
\[|t|^{-k}< |F^*(v\b w,-t)|, |H^*(v\b w,-t)|<|t|^{k}.\]
\end{itemize}
Note that it is enough to consider the case when $\Sigma$ is contained in a half-plane $\cH$ centered on some ray in $\Sigma(0)$, since $\Sigma$ is in any case contained in a finite union of such half-planes. Then $v,w$ and $w-v$ all lie in $-\cH$. By homogeneity of the functions $F$ and $H$ we can rotate so that $-\cH$ is the right-hand half-plane. The claims then follows from Propositions \ref{ass}, \ref{assy} and \ref{inflimit}.  
\end{proof} 

A very similar argument shows that when $z/w\in (0,1)$ the expressions \eqref{ji} give a solution to Problem \eqref{tro}. We leave the details to the reader.

\begin{remark}
Using \eqref{blob}, the unique solution to Problem \ref{tra} is 
\begin{equation}\label{argh}\begin{gathered}B_n(v,w,t)= {F}^*(v+nw \b w,-t), \\
D_n(v,w,t)=H^*(v+nw\b w,-t)\cdot {F}^*(v+nw  \b w,-t)^{n}.\end{gathered}\end{equation}
\end{remark}

\subsection{The $\tau$-function}

Let us introduce a function
\begin{equation}\label{gg}H^\dagger(z\b\omega_2,\omega_1)=H(z\b\omega_2,\omega_1) \cdot e^{R(z\b\omega_2,\omega_1)},\end{equation}
where $R$ is the expression
\begin{equation}
\label{rr}
R(z\b \omega_2,\omega_1)=\Big(\frac{\omega_1}{2\pi i \omega_2}\Big)^2\cdot \big( \Li_3(e^{2\pi i z/\omega_1})-\zeta(3)\big)+\frac{i\pi}{12}\cdot\frac{ z}{\omega_1}.\end{equation}
The point of this is that the relations  \eqref{latehome} and \eqref{nosven} become
\begin{equation}
\label{latehome2}\frac{\partial}{\partial \omega_2}\log F^*(z\b\omega_1,\omega_2)=\frac{\partial}{\partial z} \log H^\dagger(z\b\omega_2,\omega_1),\end{equation}
\begin{equation}
\label{nosven2}
\frac{\partial}{\partial \omega_2} \log H^*(z\b\omega_1,\omega_2)=\frac{\partial}{\partial \omega_1}\log H^\dagger(z\b\omega_2,\omega_1).\end{equation}


We can use these relations to write down a $\tau$-function $\tau_n(v,w,t)$ for the family of solutions \eqref{argh}. Such a function is uniquely defined up to multiplication by a nonzero constant. We claim that a possible choice is\begin{equation}
 \label{rang}\tau_n(v,w,t)= H^\dagger(v+nw\,|-t,w).\end{equation}
Indeed,  \eqref{rang} is homogeneous under rescaling all variables and, recalling  the definition of the central charge $Z$ from Section \ref{skip}, the required  relations \eqref{hungry} become the identities
\[\frac{\partial}{\partial t}\log {F^*}(v+nw \b w,-t)=-\frac{\partial}{\partial v}\log H^\dagger(v+nw\,|-t,w),\]
 \begin{equation*}{\frac{\partial}{\partial t}\Big(\log {H^*}(v+nw \b w,-t) +n \log {F^*}(v+nw \b w,-t)\Big)}=-\frac{\partial}{\partial w}\log H^\dagger(v+nw\,| -t,w),\end{equation*}
which follow easily from \eqref{latehome2} and \eqref{nosven2}.

In the case $n=0$, combining the definition \eqref{defh} with the integral representation \eqref{homealone} easily gives the integral representation \eqref{alba} for the function $K(v,w,t)=H(v\,| -t,w)$. 
Comparing \eqref{rr} with \eqref{gass} and \eqref{gassi2}    shows that as $t\to 0$ there is an asymptotic expansion
\[ \log \tau(v,w,t)\,\sim\, \ -\frac{1}{12}\log \Big(\frac{-w}{t}\Big) + \frac{i\pi}{12}\cdot \frac{v}{w}\]\[+
\sum_{g\geq 1}   \frac{B_{2g} \cdot \Li_{3-2g}(e^{2\pi i v/w})}{2g\cdot (2g-2)!  }\, \Bigg(\frac{2\pi i t}{w}\bigg)^{2g-2}  +\sum_{g\geq 2}\frac{ B_{2g} \cdot B_{2g-2}  }{2g\cdot (2g-2)\cdot  (2g-2)! }\, \bigg(\frac{2\pi i t}{w}\bigg)^{2g-2}
.\]
This completes the proof of Theorem \ref{corr}.

\begin{appendix}
\section{Stability conditions and  DT theory for sheaves of dimension $\leq 1$}
\label{app}
In this Appendix we recall some results on stability conditions on categories of coherent sheaves supported in dimension $\leq 1$ and the corresponding Donaldson-Thomas invariants. We start with the case of a  projective Calabi-Yau threefold, before considering the case of the resolved conifold. We shall not give full proofs for the results on stability conditions stated here, since they are quite standard, and not logically necessary for the main results of the paper.  For more on stability conditions in general the reader can consult \cite{Br1,Br2}, while the particular stability conditions considered here are  studied in detail in \cite[Section 7]{To}.

\subsection{Stability conditions on a Calabi-Yau threefold}

 Let $X$ be a smooth, projective complex variety of dimension three, with trivial canonical bundle $\omega_X\isom \O_X$. Let $\Coh(X)$ denote the abelian category of coherent sheaves on $X$, and let\[\cA=\Coh_{\leq 1}(X)\subset \Coh(X)\]
 be the full subcategory consisting  of sheaves whose set-theoretic support  has dimension $\leq 1$. Any  sheaf $E\in \cA$ has a Chern character \[\ch(E)=(0,0,\ch_2(E),\ch_3(E))\in \bigoplus_{i=0}^3H^{2i}(X,\bZ),\]
 which via Poincar{\'e} duality we can view as an element
\[\ch(E)=(\beta,n)\in \Gamma=H_2(X,\bZ)\oplus \bZ.\]
This defines a group homomorphism $\ch\colon K_0(\cA)\to \Gamma$.

Let $\cK_{\bC}\subset H^2(X,\bC)$ be the complexified K{\"a}hler cone of $X$. By definition, it consists of classes of the form $\omega_\bC=B+i\omega$ with $B\in H^2(X,\bR)$ arbitrary and $\omega\in H^2(X,\bR)$ a K{\"a}hler class. Given such a complexified K{\"a}hler class $\omega_\bC\in \cK_\bC$ we define the corresponding central charge to be
the group homomorphism
\[Z_{\omega_\bC}\colon K_0(\cA)\to \bC, \quad Z(E)=\omega_\bC\cdot \ch_2(E) -\ch_3(E).\]
The assumption that $\omega$ is  K{\"a}hler  ensures that for any nonzero object $E\in \cA$ the complex number $Z(E)\in \bC$ lies in the semi-closed upper half-plane
\[\bar{\cH}= \{z=r\exp(i\pi \phi): r\in \bR_{>0}\text{ and }0<\phi\leq 1\}\subset \bC^*.\]
This is precisely the statement that $Z$  defines a stability condition on the abelian category $\cA$.

Let $\cD\subset \cD^b\Coh(X)$ be the full triangulated subcategory of the bounded derived category of coherent sheaves on $X$ consisting of complexes whose cohomology sheaves have set-theoretic support of dimension $\leq 1$. The standard t-structure on $\cD^b\Coh(X)$ induces a bounded t-structure on the triangulated category $\cD$ whose heart can be identified with $\cA$. General results give the existence of a complex manifold $\Stab(\cD)$ parameterising stability conditions on $\cD$ whose central charge factors via the Chern character $\ch\colon K_0(\cD)\to \Gamma$. It comes with a natural action of the group of triangulated auto-equivalences $\Aut(\cD)$ of the  category $\cD$, and a period map
\begin{equation}
\label{peri}\pi\colon \Stab(\cD)\to \Hom_{\bZ}(\Gamma,\bC), \quad (Z,\cP)\mapsto Z,\end{equation}
which is a local analytic isomorphism. There is also a natural action of $\bC$ on $\Stab(\cD)$ which acts on central charges by rotation. 
The above construction gives

\begin{prop}
\label{peewee}
There is an open subset $U(X)\subset \Stab(\cD)$ with the following properties:
\begin{itemize}
\item[(i)] It is preserved by the action of $\bC$ and hence also the double shift $[2]\in \Aut \cD(X)$.
\smallskip
\item[(ii)] The restriction of the period map \eqref{peri} to $U(X)$ is the universal cover of the  image of the open embedding
\begin{equation}
\label{chop}\cK_{\bC}\times \bC^*\subset \Hom_{\bZ}(\Gamma,\bC), \qquad    (\omega_\bC,q)\mapsto q\cdot  Z_{\omega_\bC},\end{equation}
and the covering group is realised by the  action of the double shift $[2]$.
\smallskip

\item[(iii)] All points of $U(X)$ are obtained by applying the action of $\bC$ to a unique stability condition with heart $\cA\subset \cD$ arising from the construction above.
\end{itemize}
\end{prop}

The BPS invariants for stability conditions lying in the subset $U(X)\subset \Stab(\cD)$ were studied by Joyce and Song \cite{JS} (see particularly Sections 6.3 and 6.4). More precisely, they considered the stability conditions on the abelian category $\cA$ described above. They showed that 
\[\Omega(0,n)=-\chi(X)\]
for all integers $n\in \bZ\setminus\{0\}$. They also conjectured that for any effective class $\beta\in H_2(X,\bZ)$ the BPS invariant
\[\Omega(\beta,n)=\GV(0,\beta)\]
coincides with the genus 0 Gopakumar-Vafa invariant counting genus 0 curves in $X$. This is now known to hold in many examples, but not in general. In particular, it is not known in general whether the numbers $\Omega(\beta,n)\in \bQ$ are independent of $n$, or whether they are integral.

\subsection{Stability conditions on the resolved conifold}

The material of the last subsection can be generalised to quasi-projective Calabi-Yau threefolds with some extra assumptions. For example, in \cite[Section 6.7]{JS} Joyce and Song require  that $X$ is compactly embeddable. Here we just consider the case of the resolved conifold $X=\operatorname{Tot} \O_{\bP^1}(-1)^{\oplus 2}$. 

The variety $X$ contains a unique compact curve, namely the zero section $C\isom \bP^1\subset X$. It defines a class $\beta=[C]\in H_2(X,\bZ)$. There are identifications
\[\Gamma=H_2(X,\bZ)\oplus H_0(X,\bZ)=\bZ\cdot \beta \oplus \bZ \cdot \delta.\]
We let $\cA\subset \Coh(X)$ be the full subcategory consisting of objects with compact support. All such objects are supported in  dimension $\leq 1$.
The Chern character map together with Poincar{\'e} duality gives a group homomorphism
\[\ch=(\ch_2,\ch_3)\colon K_0(\cA) \to \Gamma.\]
We write $\O_C(n)$ for the degree $n$ line bundle supported on the rational curve $C\subset X$, and $\O_x$ for the skyscraper sheaf supported at a point $x\in X$. We choose the sign of the generator $\delta$ so that
\[\ch(\O_C(n))=\beta-n\delta, \qquad \ch(\O_x)=-\delta.\]

Let $\cD^b\Coh(X)$ denote the bounded derived category of coherent sheaves on $X$, and let $\cD\subset \cD^b\Coh(X)$ denote the full triangulated subcategory consisting of objects whose support is compact. As before, the standard t-structure on $\cD^b\Coh(X)$ restricts to give a bounded t-structure on the category $\cD$ whose heart can be identified with $\cA\subset \cD$.
We can define an open subset $U(X)\subset \Stab(\cD)$ exactly as in Proposition \ref{peewee}. We denote by $\Stab^0(\cD)\subset \Stab(\cD)$ the connected component containing it.

\begin{thm}
The period  map \eqref{peri} restricted to the connected component $\Stab^0(\cD)$
is a regular covering map over its image, which is the open subset
\[M=\{Z\colon \Gamma\to \bC: Z(\beta+n\delta)\neq 0, Z(\delta)\neq 0\}.\]
The deck transformations can be identified with the subgroup of auto-equivalences of $\cD$ generated by the spherical twists in the objects $\O_C(n)$ for $n\in \bZ$ together with the second shift $[2]$.
\end{thm}

There is a projection $\pi\colon X\to \bP^1$, and pulling back the line bundle $\O(1)$ gives a line bundle on $X$ which we also denote $\O(1)$. Tensoring with this generates a subgroup of $\Aut(\cD)$ isomorphic to $\bZ$ . The element $n\in \bZ$ acts on $\Gamma$ via
\[\beta\mapsto \beta-n\delta, \qquad \delta\mapsto \delta.\]
This is the origin of the $\bZ$-symmetry used in Section \ref{diffrence}.

Let us denote by $\Aut^0(\cD)$ the group of auto-equivalences of $\cD$ generated by spherical twists in the objects $\O_C(n)$, the double shift $[2]$, together with tensoring with $\O(1)$.
The standard action of $\bC$ on $\Stab(\cD)$ descends to an action of $\bC^*$ on $\Stab(\cD)/[2]$ which at the level of central charges is the obvious rescaling  action. Taking the double quotient gives
\[\bC^*\backslash \Stab^0(\cD)/\Aut^0(\cD)=\bC^*\backslash M/\bZ =\bC \bP^1\setminus \{0,1,\infty\},\]
where the natural co-ordinate on $\bC\bP^1$ is $x=\exp(2\pi i Z(\beta)/Z(\delta))$.
This can be thought of as the stringy K{\"a}hler moduli space for the conifold. Two of the missing points  correspond to large volume limits in the two small resolutions of the threefold ordinary double point, and the other is a conifold point where the mass of a spherical object becomes zero.

The BPS invariants for stability conditions in the open subset $U(X)$ were  calculated by Joyce and Song \cite[Example 6.30]{JS}. One has
\[\Omega(\gamma)=\begin{cases} 1 &\text{if }\gamma=\pm \beta+n\delta \text{ for some } n\in \bZ,\\
-2 &\text{if }\gamma=k\delta\text{ for some }k\in \bZ\setminus\{0\},\end{cases}\]
with all others being zero.
\end{appendix}




\begin{thebibliography}{100}




\bibitem{Barnes2} E.W. Barnes, The genesis of the double gamma functions, Proc. London Math. Soc. 31 (1900), 358--381.

\bibitem{Barnes3} E.W. Barnes, The theory of the double gamma function,  Philos. Trans. Roy. Soc. A 196 (1901), 265--388.

\bibitem{Barnes4} E.W. Barnes, On the theory of the multiple gamma function, Trans. Camb. Philos. Soc. 19 (1904), 374--425.

\bibitem{Br1} T. Bridgeland, Stability conditions on triangulated categories. Ann. of Math. (2) 166 (2007), no. 2, 31--345.

\bibitem{Br2} T. Bridgeland, Spaces of stability conditions, Algebraic geometry--Seattle 2005. Part 1, 1--21, 
Proc. Sympos. Pure Math., 80, Part 1, Amer. Math. Soc., Providence, RI, 2009.

\bibitem{RHDT} T. Bridgeland, Riemann-Hilbert problems from Donaldson-Thomas theory, arxiv:1611.03697.

\bibitem{BTL} T. Bridgeland and V. Toledano Laredo, Stability conditions and Stokes factors. Invent. Math. 187 (2012), no. 1, 61--98.

\bibitem{COGP} P. Candelas, X. de la Ossa, P. Green and L. Parkes, A pair of Calabi-Yau manifolds as an exactly soluble superconformal theory. Nuclear Phys. B 359 (1991), no. 1, 21--74.

\bibitem{Do} M.R. Douglas,  D-Branes on Calabi-Yau manifolds. European Congress of Mathematics, Vol. II (Barcelona, 2000), 449--466, Progr. Math., 202, Birkh\"auser, Basel, 2001. 


 \bibitem{FP} C. Faber and R. Pandharipande, Hodge integrals and Gromov-Witten theory. Invent. Math. 139 (2000), no. 1, 173--199.

\bibitem{FG} V.V. Fock and A.B. Goncharov, The quantum dilogarithm and representations of quantum cluster varieties, Invent. Math. 175 (2009), no. 2, 223--286. 



\bibitem{GMN1} D. Gaiotto, G. Moore and A. Neitzke, Four-dimensional wall-crossing via three-dimensional field theory. Comm. Math. Phys. 299 (2010), no. 1, 163--224.

%
\bibitem{Giv} A. Givental, Equivariant Gromov-Witten invariants, Internat. Math. Res. Notices 1996, no. 13, 613--663.


\bibitem{HO} Y. Hatsuda and K. Okuyama, Resummations and non-perturbative corrections, J. High Energy Phys. 2015, no. 9, 051, 28 pp.


\bibitem{JM} M. Jimbo and T. Miwa, Quantum KZ equation with $|q| =1$ and correlation functions of the XXZ model in the gapless regime, J. Phys. A 29 (1996), no. 12, 2923--2958. 


\bibitem{JS} D. Joyce and Y. Song, A theory of generalized Donaldson-Thomas invariants. Mem. Amer. Math. Soc. 217 (2012), no. 1020, 199 pp.

\bibitem{KLS} S. Kharchev, D. Lebedev, M. Semenov-Tian-Shansky, Unitary representations of $U_q(\mathfrak{sl}(2,\bR))$, the modular double, and the multiparticle $q$-deformed Toda chains, Comm. Math. Phys. 225 (2002), no. 3, 573--609. 

\bibitem{KKP} L. Katzarkov, M. Kontsevich and T. Pantev, Hodge theoretic aspects of mirror symmetry. From Hodge theory to integrability and TQFT $tt^*$-geometry, 87--174, Proc. Sympos. Pure Math., 78, Amer. Math. Soc., Providence, RI, 2008.

\bibitem{KS1} M. Kontsevich and Y. Soibelman, Stability structures, motivic Donaldson-Thomas invariants and cluster transformations, arxiv 0811.2435.


\bibitem{Ko} S. Koshkin, Quantum Barnes function as the partition function of the resolved conifold, Int. J. Math. Math. Sci. 2008,  47 pp.


\bibitem{KK}  S. Koyama and  N. Kurokawa, Multiple sine functions, Forum Math. 15 (2003) 839--876.

\bibitem{KM} D. Krefl and R. Mkrtchyan, Exact Chern-Simons / topological string duality, J. High Energy Phys. 2015, no. 10, 045, 26 pp.


\bibitem{MNOP} D. Maulik, N. Nekrasov, A. Okounkov and R. Pandharipande, Gromov-Witten theory and Donaldson-Thomas theory. I. Compos. Math. 142 (2006), no. 5, 1263--1285.



\bibitem{N} A. Narukawa, The modular properties and the integral representations of the multiple elliptic gamma functions, Adv. Math. 189 (2004), no. 2, 247--267.

\bibitem{ON}  N. Nekrasov and A. Okounkov, Seiberg-Witten theory and random partitions. The unity of mathematics, 525--596, Progr. Math., 244, Birkh{\"a}user, Boston, MA, 2006.

\bibitem{P} R. Pandharipande, Three questions in Gromov-Witten theory, Proceedings of the International Congress of Mathematicians, Vol. II (Beijing, 2002), 503--512, Higher Ed. Press, Beijing, 2002.

\bibitem{PP} R. Pandharipande and A. Pixton, Gromov-Witten/Pairs correspondence for the quintic 3-fold. J. Amer. Math. Soc. 30 (2017), no. 2, 389--449.


\bibitem{Ruj1} S.N.M. Ruijsenaars,  Special functions defined by analytic difference equations, Special functions 2000: current perspective and future directions (Tempe, AZ), 281--333, NATO Sci. Ser. II Math. Phys. Chem., 30, Kluwer Acad. Publ., Dordrecht, 2001.


\bibitem{Ruj2} S.N.M. Ruijsenaars, On Barnes' multiple zeta and gamma functions. Adv. Math. 156 (2000), no. 1, 107--132.


\bibitem{Ruj3} S.N.M. Ruijsenaars, First order analytic difference equations and integrable quantum systems. J. Math. Phys. 38 (1997), no. 2, 1069--1146.


\bibitem{To} Y. Toda, Stability conditions and crepant small resolutions,  Trans. Amer. Math. Soc. 360 (2008), no. 11, 6149--6178.

\bibitem{WW} E.T. Whittaker and G.N. Watson, A course of modern analysis. An introduction to the general theory of infinite processes and of analytic functions; with an account of the principal transcendental functions. Reprint of the fourth (1927) edition. Cambridge University Press, 1996. vi+608 pp.

\end{thebibliography}
\end{document}